\documentclass[12pt]{amsart}
\usepackage[utf8]{inputenc}
\usepackage[T1]{fontenc}
\usepackage{lmodern}
\usepackage{color,graphicx,array, amssymb, amscd,slashed,bm}
\usepackage{comment}
\usepackage{hyperref}
\usepackage{graphicx}
\usepackage{subcaption}
\newtheorem{Theorem}{Theorem}[section]
\newtheorem{Lemma}[Theorem]{Lemma}
\newtheorem{Proposition}[Theorem]{Proposition}
\newtheorem{Corollary}[Theorem]{Corollary}
\theoremstyle{definition}
\newtheorem{Definition}[Theorem]{Definition}

\newtheorem{Remark}[Theorem]{Remark} 

\numberwithin{equation}{section}
\newcommand{\R}{\mathbb R}

\newcommand{\C}{\mathbb C}

\newcommand{\Uone}{{\rm U}_1}
\newcommand{\id}{\operatorname{Id}}

\newcommand{\la}{\langle}
\newcommand{\ra}{\rangle}

\newcommand{\SL}{{\rm SL}(2, \mathbb C)}

\newcommand{\SU}{{ \rm SU(1,1)}}
\newcommand{\SO}{{ \rm SO_{0}(2,2)}}
\newcommand{\U}{{ \rm U(1)}}

\newcommand{\di}{\operatorname{diag}}

\renewcommand{\Re}{\operatorname {Re}}
\renewcommand{\Im}{\operatorname {Im}}

\newcommand{\red}[1]{{\leavevmode\color{red}{#1}}} 

\renewcommand{\red}[1]{#1} 
\usepackage{amsmath}	
\usepackage{cases}
\begin{document}

\title[]{Minimal Lagrangian surfaces in 
the two-dimensional complex hyperbolic quadric via the 
loop group method}
 \author[S.-P.~Kobayashi]{Shimpei Kobayashi}
 \address{Department of Mathematics, Hokkaido University, 
 Sapporo, 060-0810, Japan}
 \email{shimpei@math.sci.hokudai.ac.jp}
 \thanks{The first named author is partially supported by JSPS 
 KAKENHI Grant Number JP22K03304.}
 
 \author[S.~Zeng]{Sihao Zeng}
 \address{Department of Mathematics, Hokkaido University, 
 Sapporo, 060-0810, Japan}
 \email{sihao.zeng.f1@elms.hokudai.ac.jp}
 \thanks{The second named author is supported by Special Program: International Graduate Course for Data-driven and Hypothesis-driven Science (IGC-DHS). Corresponding author.}

 \date{\today}
\pagestyle{plain}
\begin{abstract}
We study minimal Lagrangian surfaces in the complex hyperbolic quadric.
We show that minimality of a Lagrangian surface is characterized by a loop of flat connections, which yields
an associated $\mathbb S^1$-family of isometric deformations. 
We also establish a correspondence with spacelike maximal surfaces in anti-de Sitter
$3$-space via the Gauss map.
Using the resulting harmonic map into the hyperbolic two-space, we develop a DPW-type
representation and construct explicit examples, including $\mathbb{R}$-equivariant and
radially symmetric surfaces.
In particular, under suitable conditions, the $\mathbb{R}$-equivariant family contains 
catenoid-type examples.
\end{abstract}

\maketitle

\section*{Introduction}
The complex hyperbolic quadric $Q_2^*$ in the complex anti-de Sitter $3$-space $\mathbb{CH}^3_1$ provides a natural non-compact setting for the study of Lagrangian surface geometry. Besides its role as the target of the Gauss map of a spacelike maximal surface in anti-de Sitter $3$-space $\mathbb H_1^3$ \cite{AiyamaAkutagawaWan2000, KrasnovSchlenker2007},  it carries a natural intrinsic theory of minimal Lagrangian surfaces
\cite{GVWX}.  The purpose of this paper is to show that minimal Lagrangian surfaces in $Q_2^*$ admit a natural integrable-systems description, including a loop of flat connections, an associated $\mathbb S^1$-family of isometric deformations, and a DPW-type representation.

Our first main result is a flatness characterization of minimality of a Lagrangian surface. More precisely, for a Lagrangian surface in $Q_2^*$, we construct a natural family of connections $\mathrm{d}+\omega^\lambda$ depending on a loop parameter $\lambda \in \mathbb S^1$, and prove that the surface is minimal if and only if the family $\mathrm{d}+\omega^\lambda$ is flat for all $\lambda \in \mathbb S^1$. This places minimal Lagrangian surface theory in $Q_2^*$ into the standard framework of integrable surface geometry and yields, in particular, an associated $\mathbb S^1$-family of isometric minimal Lagrangian deformations.

A second main result is the correspondence with spacelike maximal surfaces in anti-de Sitter space $\mathbb H_1^3$ and a DPW-type representation. We establish a correspondence between minimal Lagrangian surfaces
in \(Q_2^*\) and spacelike maximal surfaces in \(\mathbb H^3_1\), up to the natural
ambiguities of the lift and the normal.  Through the Lie group isomorphism between the indefinite Lie groups,
the flatness equations reduce to the harmonic map equation into the hyperbolic plane $\mathbb H^2$. This reduction leads naturally to a DPW-type representation \cite{DPW} for minimal Lagrangian surfaces in $Q_2^*$.
As an application of this representation, we construct several explicit examples, including $\mathbb R$-equivariant and radially symmetric minimal Lagrangian surfaces. In particular, under suitable closing conditions, the $\mathbb R$-equivariant family produces \emph{catenoid-type} examples, giving a natural explicit family of nontrivial annulus-type minimal Lagrangian surfaces in $Q_2^*$.

The present paper should be viewed as the non-compact counterpart of our earlier work on minimal Lagrangian surfaces in the complex quadric $Q_2$ \cite{KZ2025}, but the transition from
\[
Q_2 \cong \mathbb S^2 \times \mathbb S^2
\qquad\text{to}\qquad
Q_2^* \cong \mathbb H^2 \times \mathbb H^2
\]
is far from formal. In the non-compact case, the ambient geometry naturally brings in a Lorentzian companion space, namely anti-de Sitter $3$-space, and the relevant companion surfaces are spacelike maximal. At the same time, the reduced equations acquire a genuinely different sign structure. As a result, the integrable-systems picture in $Q_2^*$ is not a trivial modification of the compact case, but requires a separate treatment.

The paper is organized as follows.
In Section~\ref{sc:1} we develop the surface theory of Lagrangian immersions in $Q_2^*$ and
establish the flatness characterization and the correspondence with spacelike maximal surfaces
in anti-de Sitter $3$-space.
In Section~\ref{sc:DPW} we reformulate the theory in terms of harmonic maps into $\mathbb H^2$ and develop the
DPW representation.
Section~\ref{sc:Ex} is devoted to explicit examples. In Appendix \ref{app: CU}, we compare
our formulation with that of \cite{GVWX}.

\section{Minimal Lagrangian surfaces in $Q_2^{*}$, the family of flat connections, and spacelike maximal surfaces in $\mathbb{H}^{3}_{1}$}\label{sc:1}
In this section, we study the minimality of a Lagrangian immersion $f : M \to Q_{2}^{*}$, which is equivalent to the flatness of an associated family of connections  (Theorem \ref{Thm1}), and we show the existence of a one-parameter family of minimal Lagrangian surfaces (Theorem \ref{thm2}). Moreover, Theorem~\ref{Thm corresponding} establishes a correspondence between minimal Lagrangian surfaces in $Q_{2}^{*}$ and spacelike maximal surfaces in $\mathbb{H}^{3}_{1}$.

\subsection{Lagrangian surfaces in $Q_2^{*}$}\label{sub: Lagrangian surfaces}
 Let $( N , \omega )$ be a K\"ahler manifold of $\text{dim}_{\mathbb{C}}N = n$ with the 
 K\"ahler form $\omega$. An immersion $f : M \to N$ from an $m$-dimensional manifold $M$ into $N$ is said to be $\textit{totally~real}$, or $\textit{isotropic}$, if $f^{\ast}\omega = 0$. 
 In particular, a totally real immersion $f$ is said to be $\textit{Lagrangian}$ if $m = n$.

 Let $\mathbb{C}^{n}_{m}$ be the complexification of the pseudo-Euclidean space $\mathbb{R}^{n}_{m}$ with the complex bilinear form
$\la \,,\, \ra$ defined by
 \begin{equation}\label{eq:sclar}
     \left \la \boldsymbol{z} , \boldsymbol{w} \right \ra = -z_{1}w_{1}- \cdots - z_{m}w_{m} + z_{m+1}w_{m+1} + \cdots + z_{n}w_{n},
 \end{equation}
 where $\boldsymbol{z} = \left( z_{1}, \ldots, z_{n} \right), \, \boldsymbol{w} = \left( w_{1}, \ldots, w_{n}\right) \in \mathbb{C}^{n}_{m}$.
 The standard Hermitian inner product $(~,~)$ on $\mathbb{C}^{n}_{m}$ is given by $\left( \boldsymbol{z}, \boldsymbol{w} \right) = \left \la \boldsymbol{z} , \bar{\boldsymbol{w}} \right \ra$, where $\bar{\boldsymbol{w}}$ is the complex conjugate of $\boldsymbol{w}$. 
 The complex pseudo-hyperbolic space $\mathbb{C H}^{n}_{m}$ of dimension $n$
 is a space of negative lines in the complex projective space $\mathbb{C }P^{n}$ defined by
 \begin{equation*}
     \mathbb{C H}^{n}_{m} := \left \{ [\boldsymbol{z}] \in \mathbb{C }P^{n} \mid 
     \mbox{$( \boldsymbol{z} , \boldsymbol{z} ) <0$ for $\boldsymbol{z} \in \C^{n+1}_{m+1}$}  \right \}.
 \end{equation*}
 On the one hand, the pseudo-hyperbolic space of constant sectional curvature $c<0$
 \[\mathbb{H}^{2n+1}_{2m+1}(c)  := \{ \boldsymbol{z} \in \mathbb{C}_{m+1}^{n+1} \mid ( 
 \boldsymbol{z},  \boldsymbol{z} ) = 1/c\}
 \] 
 has  a natural indefinite Hopf fibration 
\[\mathbb{H}^{2n+1}_{2m+1}(c) \to   \mathbb{C H}^{n}_{m},\]
 see \cite{anciaux2010minimal}.
 Thus $\mathbb{C H}^{n}_{m}$ can also be realized as  the quotient of 
 by the timelike $\mathbb S^1$-fiber and the holomorphic sectional curvature of $\mathbb{C H}^{n}_{m}$ is $4 c<0$.
 In particular, for $n=1, m = 0$, $\mathbb{H}^{3}_{1}(c)$
 is the anti-de Sitter $3$-space,
 and for $n=3, m = 1$, $ \mathbb{C H}^{3}_{1}$ is the complex anti-de Sitter $3$-space.
 In this paper, we denote by $\mathbb{H}^{3}_{1}$ the anti-de Sitter $3$-space with constant sectional curvature $c = -1$.
  
The complex hyperbolic quadric $Q_2^{*}$ is realized as  
\begin{equation}\label{eq:Q2}
Q_{2}^{*} := \left \{ \left [ \boldsymbol{z} \right ] \in \mathbb{C}\mathbb{H}^{3}_{1} \mid \la \boldsymbol{z}, \boldsymbol{z} \ra = 0 \right \}^{0},
\end{equation}
where the superscript $0$ denotes a connected component. 
It is known that $Q_2^{*}$ is a homogeneous K\"ahler-Einstein manifold and it is isometric to $\mathbb H^2 \times  \mathbb H^2$, 
 where the curvatures of the two hyperbolic planes $\mathbb H^2$ are normalized to $-4$, \cite{GVWX, WV2021}.

 Let $f : M \to Q_{2}^{*}$ be a  conformal Lagrangian immersion from a Riemann surface $M$ into $Q_{2}^{*}$. Thus there exists a simply connected domain $\mathbb{D} \subset M$ with conformal coordinate $z = x + iy$, and the induced metric on $\mathbb{D}$ can be computed as 
 \begin{equation*}
 ds^{2}_{M} = 2e^{u}\text{d}z\text{d}\bar{z}. 
 \end{equation*}
 Let $\mathfrak{f}: \mathbb{D} \to \mathbb{H}^{7}_{3}(-1) \subset \mathbb C^4_{2}$ be a local lift of $f$, i.e. $f = \pi \circ \mathfrak{f}$, where $\pi : \mathbb{H}^{7}_{3}(-1) \to \mathbb{C}\mathbb{H}^{3}_{1}$ is the indefinite Hopf fibration. In fact, the projection $f$ can be realized as $[\mathfrak{f}]$.
 Since $f$ is conformal and Lagrangian, and satisfies $f(M) \subset Q_{2}^{*}$, we obtain
 \begin{equation}\label{eq: conformal and Lagrangian}
 \langle  \mathfrak{f}_{z}, \overline{\mathfrak{f}_{\bar{z}}} \rangle  = 0, \quad \langle \mathfrak{f}_{z}, \overline{\mathfrak{f}_{z}} \rangle = \langle \mathfrak{f}_{\bar{z}}, \overline{\mathfrak{f}_{\bar{z}}} \rangle = e^{u}, \quad \langle  \mathfrak{f}, \mathfrak{f} \rangle  = 0, \quad \langle \mathfrak{f}_{z}, \mathfrak{f}  \rangle = \left\la \mathfrak{f}_{\bar{z}}, \mathfrak{f} \right\ra = 0.
 \end{equation}
 Here $\partial_z = \tfrac12 (\partial_x - i \partial_y)$ and 
 $\partial_{\bar z} = \tfrac12 (\partial_x + i \partial_y)$ are the complex differentiations.
 If a local lift $\mathfrak{f}$ satisfies 
 \begin{equation*}
     \langle \mathfrak{f}_{z} , \bar{\mathfrak{f}} \rangle = \langle \mathfrak{f}_{\bar{z}} , \bar{\mathfrak{f}} \rangle = 0, 
 \end{equation*}
  then we call $\mathfrak{f}$ a $\textit{horizontal~lift}$.
  
 Since horizontal lifts of $f$ are not unique, we fix one horizontal lift $\mathfrak f$ for the time being.
Define the indefinite special orthogonal group
\begin{equation*}
 \mathrm{SO}(2,2) := \left\{  A \in M_{4 \times 4} (\R) \mid  A^T \eta A = \eta,~
 \det A =1\right\}, \quad \eta =  \di(-1,-1,1,1).
\end{equation*}
Since the identity component of $\mathrm{SO}(2,2)$ acts transitively on $Q_2^{*}$ by orientation-preserving isometry, $Q_{2}^{*}$ is isomorphic to the symmetric space: 
 \begin{equation}\label{eq:symQ2}
 Q_2^{*} = \SO/ \left( \rm{SO}(2) \times \rm{SO}(2) \right),
 \end{equation}
  where the subscript $0$ denotes the identity component, see \cite{AA1998,GVWX}. Indeed, by choosing the 
  involution $\sigma = \operatorname{Ad} \operatorname{diag}(1, 1, -1, -1)$ on $\SO$,
  the fixed point set of $\sigma$ is exactly $ \rm{SO}(2) \times \rm{SO}(2)$. 
    Let $f : M \to Q_2^{*}=\SO/ \left( \rm{SO}(2) \times \rm{SO}(2) \right)$ and let $\mathcal{F}: \mathbb{D} \subset M \to \SO$ be a local lift of $f$ as
 \begin{equation}\label{eq: Frame identity component}
     \mathcal{F}:= \left( \frac{1}{\sqrt{2}}\left( \mathfrak{f} + \bar{\mathfrak{f}} \right), -\frac{i}{\sqrt{2}}\left( \mathfrak{f} - \bar{\mathfrak{f}} \right), \frac{\mathfrak{f}_{z} + \overline{\mathfrak{f}_{z}}}{\sqrt{2e^{u} + \alpha + \bar{\alpha}}}, -\frac{i\left \{\mathfrak{f}_{z}\left( e^{u} + \bar{\alpha} \right) - \overline{\mathfrak{f}_{z}}\left( e^{u} + \alpha \right) \right \} }{\sqrt{\left( 2e^{u} + \alpha + \bar{\alpha} \right)\left( e^{2u} - \alpha\bar{\alpha} \right)}} \right),
 \end{equation}
 such that $\mathcal{F}( z_{0}) = \mathrm{Id}$, where $\mathfrak{f}$ is a horizontal lift defined above and 
 \begin{equation}\label{eq:alpha}
 \alpha := \la \mathfrak{f}_{z}, \mathfrak{f}_{z} \ra. 
 \end{equation}
 By direct computation $e^{2u} - \alpha\bar{\alpha} \geq  0$ and 
 we assume that $e^{2u} - \alpha\bar{\alpha} > 0$ for the time being so that 
 the frame $\mathcal F$ is well-defined.
 Note that this assumption is not necessary for minimal Lagrangian surfaces, as we will see later,
 see Remark \ref{Rm:assumption}.
 Its Maurer-Cartan form can be computed as follows:
 \begin{equation}
     \omega = \mathcal{F}^{-1}\text{d}\mathcal{F} =  \mathcal{F}^{-1}\mathcal{F}_{z}\text{d}z + \mathcal{F}^{-1}\mathcal{F}_{\bar{z}}\text{d}\bar{z} = \mathcal{U}\text{d}z + \mathcal{V}\text{d}\bar{z}, 
 \end{equation}
 where \begin{equation}\label{eq:pq}
     \mathcal{U} = \begin{pmatrix}
 0 & 0 & -p_{1} & -p_{2} \\
 0 & 0 & -p_{3} & -p_{4} \\
 -p_{1} & -p_{3} & 0 & q \\
 -p_{2} & -p_{4} & -q & 0
\end{pmatrix}, \quad \mathcal{V} = \begin{pmatrix}
 0 & 0 & -\bar{p}_{1} & -\bar{p}_{2} \\
 0 & 0 & -\bar{p}_{3} & -\bar{p}_{4} \\
 -\bar{p}_{1} & -\bar{p}_{3} & 0 & \bar{q} \\
 -\bar{p}_{2} & -\bar{p}_{4} & -\bar{q} & 0
\end{pmatrix}, 
 \end{equation}
 with
 \begin{equation*}
 \begin{aligned}
		&p_{1} =  -\frac{1}{\sqrt{2}} \left( \frac{\alpha + e^{u} + \bar{\beta}}{\sqrt{2e^{u} + \alpha + \bar{\alpha}}} \right), \quad
		p_{2} = \frac{i}{\sqrt{2}} \left[ \frac{\left( \alpha\bar{\alpha} - e^{2u} \right) - \bar{\beta}\left( e^{u} + \alpha \right)}{\sqrt{\left( 2e^{u} + \alpha + \bar{\alpha} \right)\left( e^{2u} - \alpha\bar{\alpha} \right) }} \right], \\
        &p_{3} = \frac{i}{\sqrt{2}} \left( \frac{\alpha + e^{u} - \bar{\beta}}{\sqrt{2e^{u} + \alpha + \bar{\alpha}}} \right), \quad
        p_{4} = \frac{1}{\sqrt{2}} \left[ \frac{\left( \alpha\bar{\alpha} - e^{2u} \right) + \bar{\beta}\left( e^{u} + \alpha \right)}{\sqrt{\left( 2e^{u} + \alpha + \bar{\alpha} \right)\left( e^{2u} - \alpha\bar{\alpha} \right) }} \right], \\
        &q = i \left[ \frac{\frac{1}{2}e^{u}\left( \alpha_{z} - \bar{\alpha}_{z} \right) + \frac{1}{2}\left( \alpha_{z}\bar{\alpha} - \bar{\alpha}_{z}\alpha \right) - e^{u}\phi\left( 2e^{u} + \bar{\alpha} + \alpha \right) - u_{z}e^{u}\left( e^{u} + \alpha \right)}{\left( 2e^{u} + \alpha + \bar{\alpha} \right)\sqrt{e^{2u} - \alpha\bar{\alpha}}} \right], 
 \end{aligned}
\end{equation*}
  and 
  \begin{equation}\label{complex function}
\beta := \la \mathfrak{f}_{z}, \mathfrak{f}_{\bar{z}} \ra, \quad \phi := e^{-u} \la \mathfrak{f}_{z\bar{z}}, \overline{\mathfrak{f}_{\bar{z}}}\ra.
  \end{equation}
Note that $\alpha \, \text{d}z^2$, $\beta  \, \text{d}z \text{d}\bar z$ and $\phi \, \text{d}z$
 are well-defined differentials for the horizontal lift $\mathfrak f$.  
 Moreover, we can characterize the minimality of the surface by the following theorem.
\begin{Theorem}\label{Thm: minimality}
    Let $f : M \to Q_{2}^{*}$ be a conformal Lagrangian immersion and $\Phi$ be the associated one-form defined by $\Phi = \phi \mathrm{d} z$. Then $f$ is minimal if and only if $\Phi = 0$.
\end{Theorem}

\begin{proof}
       Choose an indefinite unitary $\mathrm{U}(2,2)$-frame of $\mathbb{C}^{4}_{2}$ as follows: 
    \begin{equation*}
        e_{0} = \mathfrak{f}, \quad e_{1} = \bar{\mathfrak{f}}, \quad e_{2} = e^{-u/2}\mathfrak{f}_{z}, \quad e_{3} = e^{-u/2}\mathfrak{f}_{\bar{z}}, 
    \end{equation*}
    where $\mathfrak{f}$ is the horizontal lift of $f$ and is defined in Subsection \ref{sub: Lagrangian surfaces}. The condition for $f$ to be minimal is the vanishing of the traces of second fundamental forms \cite{CW1983}, and the calculations are straightforward.
    Thus the result follows by direct computation  similar to the proof of \cite[Theorem 3.2]{WX}.
\end{proof}

The compatibility condition $\mathcal{F}_{z\bar{z}} = \mathcal{F}_{\bar{z}z}$ is equivalent to $\mathcal{U}_{\bar{z}} - \mathcal{V}_{z} = \left[ \mathcal{U}, \mathcal{V} \right]$, as well as to the Maurer-Cartan equation $\text{d}\omega + \frac{1}{2}\left[ \omega \wedge \omega \right] = 0$, which can also be interpreted as the flatness of the connection $\text{d} + \omega$. More explicitly, it reduces to the following system of equations: 
 \begin{subnumcases}{}
	&$p_{1}\bar{p}_{3} + p_{2}\bar{p}_{4} = \bar{p}_{1}p_{3} + \bar{p}_{2}p_{4}$, \label{eq:M-C1}\\
	&$\bar{p}_{2}q - p_{2}\bar{q} = p_{1\bar{z}} - \bar{p}_{1z}$, \label{eq:M-C2}\\
	&$p_{1}\bar{q} - \bar{p}_{1}q = p_{2\bar{z}} - \bar{p}_{2z}$, \label{eq:M-C3}\\
        &$\bar{p}_{4}q - p_{4}\bar{q} = p_{3\bar{z}} - \bar{p}_{3z}$, \label{eq:M-C4}\\
        &$p_{3}\bar{q} - \bar{p}_{3}q = p_{4\bar{z}} - \bar{p}_{4z}$, \label{eq:M-C5}\\
        &$p_{2}\bar{p}_{1} + p_{4}\bar{p}_{3} - \bar{p}_{2}p_{1} - \bar{p}_{4}p_{3} = -q_{\bar{z}} + \bar{q}_{z}$.  \label{eq:M-C6}
\end{subnumcases}
By direct computation, \eqref{eq:M-C1}-\eqref{eq:M-C5} can be simplified to the following system of equations:
\begin{numcases}
{}
	&$e^{2u} - \beta\bar{\beta} = \alpha \bar{\alpha}$, \label{eq: condition 1}\\
	&$\frac{1}{2}\bar{\alpha}_{\bar{z}}\alpha - e^{2u}\bar{\phi} - u_{\bar{z}}e^{2u} + \beta \bar{\beta}_{\bar{z}} = \frac{1}{2} \bar{\alpha}_{z}\beta$, \label{eq: condition 2}\\
	&$\phi \beta = \bar{\phi} \alpha + \frac{1}{2} \alpha_{\bar{z}}$. \label{eq: condition 3}
\end{numcases}
 Note that \eqref{eq:M-C6} can also be expressed by the function $\alpha$, $\beta$ and 
 $\phi$, however, the expression is rather involved and will not be used later, 
 so we omit it here.
 
\subsection{Minimality of Lagrangian surfaces and the family of flat connections}\label{sbsc: Minimality}
Since $Q_2^*$ is a symmetric space as in \eqref{eq:symQ2}, minimal surfaces in $Q_2^*$ may be viewed as conformal harmonic maps, and hence the integrable-systems approach applies.
We consider the following family of connection one-forms $\text{d} + \omega^{\lambda}$:
 \begin{equation}\label{lambda family}
     \omega^{\lambda} = \lambda^{-1}\omega^{'}_{\mathfrak{p}} + \omega_{\mathfrak{k}} + \lambda\omega^{''}_{\mathfrak{p}}, \quad 
     (\lambda \in \mathbb{S}^{1}),
 \end{equation}
where $\mathfrak{g}  = \operatorname{Lie}\left( \SO \right)= \mathfrak{so}(2,2)$ admits the decomposition  $\mathfrak{g} = \mathfrak{k} \oplus \mathfrak{p}$
with the fixed point subalgebra $\mathfrak{k} = \operatorname{Fix} (d \sigma)=  \mathfrak{so}(2) \times \mathfrak{so}(2)$ and its complement $\mathfrak p$, and $\omega_{\mathfrak{k}}$ and $\omega_{\mathfrak{p}}$ are the $\mathfrak{k}$- and the $\mathfrak{p}$-valued 1-forms.
Moreover $\prime$ and $\prime \prime$ denote the $(1,0)$- and the $(0,1)$-parts, respectively. While the flatness of $\text{d} + \omega$ corresponds to the flatness of $\left( \text{d} + \omega^{\lambda} \right)|_{\lambda = 1}$, requiring $\mathrm{d} + \omega^{\lambda}$ to be flat for all $\lambda \in \mathbb{S}^{1}$ imposes an additional condition of harmonicity on the Lagrangian surface $f$.  
More explicitly, \eqref{lambda family} can be written as follows:
\begin{equation*}
    \omega^{\lambda} = \mathcal{U}^{\lambda}\text{d}z + \mathcal{V}^{\lambda}\text{d}\bar{z},
\end{equation*}
where \begin{equation*}
     \mathcal{U}^{\lambda} = \begin{pmatrix}
 0 & 0 & -\lambda^{-1}p_{1} & -\lambda^{-1}p_{2} \\
 0 & 0 & -\lambda^{-1}p_{3} & -\lambda^{-1}p_{4} \\
 -\lambda^{-1}p_{1} & -\lambda^{-1}p_{3} & 0 & q \\
 -\lambda^{-1}p_{2} & -\lambda^{-1}p_{4} & -q & 0
\end{pmatrix}, \quad \mathcal{V}^{\lambda} = \begin{pmatrix}
 0 & 0 & -\lambda\bar{p}_{1} & -\lambda\bar{p}_{2} \\
 0 & 0 & -\lambda\bar{p}_{3} & -\lambda\bar{p}_{4} \\
 -\lambda\bar{p}_{1} & -\lambda\bar{p}_{3} & 0 & \bar{q} \\
 -\lambda\bar{p}_{2} & -\lambda\bar{p}_{4} & -\bar{q} & 0
\end{pmatrix}.
 \end{equation*}
The condition $\text{d}\omega^{\lambda} + \frac{1}{2}[ \omega^{\lambda} \wedge \omega^{\lambda} ] = 0$ is now equivalent to $\mathcal{U}^{\lambda}_{\bar{z}} - \mathcal{V}^{\lambda}_{z} = [ \mathcal{U}^{\lambda} , \mathcal{V}^{\lambda} ]$, and it is equivalent to the following system of equations:
\begin{subnumcases}{}
	&$p_{1}\bar{p}_{3} + p_{2}\bar{p}_{4} = \bar{p}_{1}p_{3} + \bar{p}_{2}p_{4}$, \label{eq:l M-C1}\\
	&$\lambda\bar{p}_{2}q-\lambda^{-1}p_{2}\bar{q} = \lambda^{-1}p_{1\bar{z}} - \lambda\bar{p}_{1z}$, \label{eq:l M-C2}\\
	&$\lambda^{-1}p_{1}\bar{q} - \lambda\bar{p}_{1}q = \lambda^{-1}p_{2\bar{z}} -\lambda\bar{p}_{2z}$, \label{eq:l M-C3}\\
        &$\lambda\bar{p}_{4}q - \lambda^{-1}p_{4}\bar{q} = \lambda^{-1}p_{3\bar{z}} - 
    \lambda\bar{p}_{3z}$, \label{eq:l M-C4}\\
        &$\lambda^{-1}p_{3}\bar{q} - \lambda\bar{p}_{3}q = \lambda^{-1}p_{4\bar{z}} -\lambda\bar{p}_{4z}$, \label{eq:l M-C5}\\
        &$p_{2}\bar{p}_{1} + p_{4}\bar{p}_{3} - \bar{p}_{2}p_{1} - \bar{p}_{4}p_{3} = -q_{\bar{z}} + \bar{q}_{z}$.  \label{eq:l M-C6}
\end{subnumcases}
By direct computation, \eqref{eq:l M-C1}-\eqref{eq:l M-C6} can be simplified to the following system of equations:
\begin{numcases}
{}
	&$e^{2u} - \beta\bar{\beta} = \alpha \bar{\alpha}$ \label{eq:1 condition 1}, \\
	&$\alpha_{\bar{z}} = \bar{\alpha}_{z} = \phi = 0$, \label{eq:1 condition 2}\\
	&$\frac{1}{2}\bar{\alpha}\alpha_{z} + \bar{\beta}\beta_{z} - u_{z}e^{2u} = 0$, \label{eq:1 condition 3}\\
        &$u_{z\bar{z}}e^{u}|\beta|^{2} - \frac{1}{4}|\alpha_{z}|^{2} e^{u} - |u_{z}|^{2} e^{u} |\alpha|^{2}  + \frac{1}{2} \alpha_{z} u_{\bar{z}} e^{u} \bar{\alpha} + \frac{1}{2}\bar{\alpha}_{\bar{z}} u_{z} e^{u} \alpha - 2|\beta|^{4} = 0$. \label{eq: second order PDE}
\end{numcases}
In particular, equation \eqref{eq: second order PDE} follows from equation \eqref{eq:l M-C6} after simplification.
Without loss of generality, we assume that $\alpha$ and $\beta$ are not identically zero by the following Lemma: 
\begin{Lemma}\label{lem:totally}
When $\alpha \equiv 0$ (resp. $\beta \equiv 0$), then the minimal Lagrangian surface in $Q_{2}^{*}$ is an open part of the diagonal surface (resp. a product of geodesics). 
\end{Lemma}
\begin{proof}
By the Gaussian curvature $K = -e^{-u}u_{z\bar{z}}$ and a straightforward computation, we obtain that $K = -2$  (resp. $K = 0$) when $\alpha \equiv 0$ (resp. $\beta \equiv 0$).
 By using Theorem 5.4 of \cite{GVWX}, the results can be obtained directly.
\end{proof}

We now characterize the minimality in terms of the family of flat
connections $\text{d} + \omega^{\lambda}$. 
\begin{Theorem}\label{Thm1}
    Let $f : M \to Q_{2}^{*}$ be a Lagrangian immersion and let $\mathrm{d} + \omega^{\lambda}$ be the family of connections in \eqref{lambda family}. Then the following statements are equivalent:  
    \begin{enumerate}
        \item The Lagrangian immersion $f$ is minimal.
        \item The connections $\mathrm{d} + \omega^{\lambda}$ are flat for all $\lambda \in \mathbb{S}^{1}$, i.e. \eqref{eq:l M-C1}-\eqref{eq:l M-C6} holds for all $\lambda \in \mathbb{S}^{1}$.
        \item The quadratic differential $\alpha \, \mathrm{d} z^2$ is holomorphic and $\varphi = \arg{(\beta)}$ is constant, where $\alpha, \beta$ are defined in \eqref{eq:alpha}, \eqref{complex function} and $\arg{(\beta)}$ denotes the argument of $\beta$.
    \end{enumerate}
\end{Theorem}

\begin{proof}  Note that the structure equations \eqref{eq: condition 1}-\eqref{eq: condition 3} and \eqref{eq:1 condition 1}-\eqref{eq:1 condition 3} have similar expressions as those in the complex quadric $Q_{2}$ case \cite{KZ2025}. In fact, the derivation is parallel to that in the compact case, although the final compatibility
equation differs by sign in the reduced term in equation \eqref{eq: second order PDE}, i.e. the sign of $2|\beta|^{4}$ is opposite. 
   
   Let us prove (1) $\Rightarrow$ (2): 
   If $f$ is minimal then $\phi =0$, then \eqref{eq: condition 3} \eqref{eq: condition 2} can be simplified  to \eqref{eq:1 condition 2} and \eqref{eq:1 condition 3}, respectively, and \eqref{eq:1 condition 1} and  
    \eqref{eq: second order PDE} are clearly satisfied.
    
      (2) $\Rightarrow$ (3): 
      
Equation \eqref{eq:1 condition 2} implies that \(\alpha_{\bar z}=0\), and hence
\(\alpha dz^2\) is holomorphic. Moreover, 
\eqref{eq:1 condition 1} and \eqref{eq:1 condition 3} imply
\[
    \beta \bar\beta_{z}= \beta_{z} \bar\beta .
\]
On the open set where \(\beta\neq0\), this is equivalent to the constancy of
\(\beta/\bar\beta\), and hence to the constancy of the phase
\(\varphi=\arg(\beta)\). Since the assertion is local, after multiplying the
horizontal lift by a constant phase, we may assume that \(\beta=|\beta|\);
the resulting formulas extend across the zero set of \(\beta\) by continuity. 

   (3) $\Rightarrow$ (1): From the derivative of \eqref{eq: condition 1} with respect to $\bar z$ 
   we have $2u_{\bar z} e^{2u} = \alpha \bar \alpha_{\bar z} + \beta_{\bar z} \bar \beta +\beta \bar \beta_{\bar z}$.
   Since $\varphi$ is constant,  $\beta_{\bar z} \bar \beta = \beta \bar \beta_{\bar z}= 
   \tfrac12 |\beta|^2_{\bar z}$. Combining the above equation with \eqref{eq: condition 2} under the holomorphicity of $\alpha$, the minimality $\phi=0$ follows.
\end{proof}

\subsection{Minimal Lagrangian surfaces and the elliptic sinh-Gordon equation}

Consider a new local horizontal lift of a minimal Lagrangian surface $f$:
\begin{equation}\label{eq:newlift}
\hat{\mathfrak{f}} = e^{-\frac{i\varphi}{2}}\mathfrak{f},
\end{equation}
 where $\varphi = \arg(\beta)$ is constant by Theorem \ref{Thm1}. The new invariants $\hat{\alpha}$ and $\hat{\beta}$ of $\hat{\mathfrak{f}}$ are given by
\begin{equation}
    \hat{\alpha} := \la  \hat{\mathfrak{f}}_{z}, \hat{\mathfrak{f}}_{z}  \ra = e^{-i\varphi}\alpha, \quad \hat{\beta} := \la \hat{\mathfrak{f}}_{z},\hat{\mathfrak{f}}_{\bar{z}} \ra = |\beta|,
\end{equation}
i.e. $|\hat{\alpha}| = |\alpha|$ and $ \hat{\beta}$ is a non-negative real function. 
By $\hat{\beta} = \sqrt{e^{2u} - |\hat{\alpha}|^{2}}$, 
all the data in Maurer-Cartan form now can be represented by $u$ and $\hat \alpha$. 

\begin{Definition}\label{def:extend}
 Denote the new frame of the lift $\hat{\mathfrak f}$ in \eqref{eq:newlift}
  of a minimal Lagrangian immersion by $\hat{\mathcal F}$. 
 Then by Theorem \ref{Thm1}, there exists a family of 
 frames $\hat{\mathcal{F}}_{\lambda}$ such that $\hat{\mathcal{F}}_{\lambda}|_{\lambda=1} = \hat{\mathcal F}$, and we call $\hat{\mathcal F}_{\lambda}$
 the \textit{extended frame}.
\end{Definition}
The new family of connection one-forms $\text{d} + \hat{\omega}^{\lambda}$ parameterized by $\lambda \in \mathbb{S}^{1}$ can be explicitly written as follows:
\begin{equation}\label{eq: M-C hat omega}
    \hat{\omega}^{\lambda} = \hat{\mathcal{F}}_{\lambda}^{-1}\text{d}\hat{\mathcal{F}}_{\lambda} = \hat{\mathcal{U}}^{\lambda}\text{d}z + \hat{\mathcal{V}}^{\lambda}\text{d}\bar{z},
\end{equation}
where
\begin{equation*}
     \hat{\mathcal{U}}^{\lambda} = \begin{pmatrix}
 0 & 0 & -\lambda^{-1}r & -\lambda^{-1}ir \\
 0 & 0 & -\lambda^{-1}ip & -\lambda^{-1}p \\
 -\lambda^{-1}r & -\lambda^{-1}ip & 0 & \hat{q} \\
 -\lambda^{-1}ir & -\lambda^{-1}p & -\hat{q} & 0
\end{pmatrix}, \quad \hat{\mathcal{V}}^{\lambda} = \begin{pmatrix}
 0 & 0 & -\lambda\bar{r} & \lambda i\bar{r} \\
 0 & 0 & \lambda i\bar{p} & -\lambda\bar{p} \\
 -\lambda\bar{r} & \lambda i\bar{p} & 0 & \bar{\hat{q}} \\
 \lambda i\bar{r} & -\lambda\bar{p} & -\bar{\hat{q}} & 0
\end{pmatrix}, 
 \end{equation*}
with
\begin{equation*}
    p= \frac{1}{\sqrt{2}}\left( \frac{\hat{\alpha}+e^{u}-\hat{\beta}}{\sqrt{2e^{u} + \hat{\alpha} + \bar{\hat{\alpha}} }} \right),~r= -\frac{1}{\sqrt{2}} \left( \frac{\hat{\alpha} + e^{u} + \hat{\beta} }{\sqrt{2e^{u} + \hat{\alpha} + \bar{\hat{\alpha}} }}\right),~\hat{q} =  i \left( \frac{e^{-u}\frac{1}{2}\hat{\alpha}_{z}\hat{\beta} - e^{-u} \hat{\alpha} \hat{\beta}_{z} - \hat{\beta}_{z}}{2e^{u} + \hat{\alpha} + \bar{\hat{\alpha}} } \right).
\end{equation*}

\begin{Remark}\label{Rm:assumption}
Since the factor $e^{2u}- \alpha \bar \alpha$ in the denominators of 
$p_1, p_2, p_3, p_4$ and $q$ cancels with the same factor in the numerators, the function 
 $p,r$ and $\hat{q}$ are well-defined even when 
 $e^{2u}- \alpha \bar \alpha = 0$. Consequently, the condition $e^{2u}- \alpha \bar \alpha > 0$ is not required for the frame of  any minimal Lagrangian surfaces.
\end{Remark}

Then the flatness condition leads to 
\begin{align}\label{eq: M-C pr}
p_{\bar{z}} = ip\bar{\hat{q}}, \quad r_{\bar{z}} = -ir\bar{\hat{q}} \quad \mbox{and}\quad
        i\frac{\hat{q}_{\bar{z}}}{2} - i\frac{\bar{\hat{q}}_{z}}{2} = |r|^{2} - |p|^{2}. 
\end{align}
Thus we have
$\left( \log |r|^{2} \right)_{z\bar{z}} = 2\hat{\beta}$, which  leads to
\begin{equation}\label{eq: sinh-Gordon equation}
    \hat{u}_{z\bar{z}} - e^{\hat{u}} + |\hat{\alpha}|^{2}e^{-\hat{u}} = 0,
\end{equation}
where $\hat{u}$ is a real function defined by 
$\hat{u} := \log (2 |r|^2)$, and 
\eqref{eq: sinh-Gordon equation} is the \textit{elliptic sinh-Gordon equation.} 
Note that we can also represent the metric on $\mathbb{D} \subset M$ by 
\begin{equation}\label{eq: metric}
    2 e^u \text{d}z\text{d}\bar{z} = (e^{\hat{u}} + |\hat{\alpha}|^{2}e^{-\hat{u}})\text{d}z\text{d}\bar{z}. 
\end{equation}
By this correspondence, $u$ satisfies \eqref{eq: second order PDE}
 if and only if $\hat u$ satisfies \eqref{eq: sinh-Gordon equation}.

\subsection{A family of minimal Lagrangian surfaces}\label{sbsc:family}
To obtain a family of minimal Lagrangian surfaces of a given minimal Lagrangian surface $f$, 
let us consider the \textit{gauge transformation} of the extended frame $\hat{\mathcal{F}}_{\lambda}$. 
For a family of smooth maps $\mathcal{G}_{\lambda}: \mathbb{D} \to \rm{SO}(2) \times \rm{SO}(2)$ 
parametrized by $\lambda \in \mathbb S^1$, let $\tilde{\mathcal{F}}_{\lambda} := \hat{\mathcal{F}}_{\lambda}\mathcal{G}_{\lambda}$. Now we explicitly choose $\mathcal{G}_{\lambda}$ as follows:
 Let $r = |r|e^{i\arg r}$, 
where $\arg r \in \mathbb{R}$ is the argument of $r$, and let 
$\lambda = e^{i \theta} (\theta \in \R)$. The gauge $\mathcal{G}_{\lambda}$ is as follows:  
\begin{equation*}
    \mathcal{G}_{\lambda} = 
    \begin{pmatrix}
 \id_{2 \times 2} & 0 \\
 0 & A
\end{pmatrix}, \quad A = \begin{pmatrix}
  \cos (\arg r- \theta) & \sin (\arg r- \theta) \\
  -\sin (\arg r -\theta) & \cos (\arg r- \theta)
\end{pmatrix} \in \rm{SO}(2).
\end{equation*}
 Then by $|r| = \frac{1}{\sqrt{2}}e^{\hat{u}/2}$ and a straightforward computation, it follows  that 
\begin{equation}\label{eq: Maurer-Cartan form tilde}
    \tilde{{\omega}}^{\lambda} := \tilde{\mathcal{F}}_{\lambda}^{-1}\text{d}\tilde{\mathcal{F}}_{\lambda} =  \mathcal{G}_{\lambda}^{-1}\hat{\omega}^{\lambda}\mathcal{G}_{\lambda} + \mathcal{G}_{\lambda}^{-1}\text{d}\mathcal{G}_{\lambda} = \tilde{\mathcal{U}}^{\lambda}\text{d}z + \tilde{\mathcal{V}}^{\lambda}\text{d}\bar{z},
\end{equation}
where
\begin{equation}\label{eq: tilde U}
    \tilde{\mathcal{U}}^{\lambda} = \frac{\sqrt{2}}{2}\begin{pmatrix}
 0 & 0 & -e^{\hat{u}/2} & -ie^{\hat{u}/2} \\
 0 & 0 & i \lambda^{-2}\hat{\alpha}e^{-\hat{u}/2} & \lambda^{-2}\hat{\alpha}e^{-\hat{u}/2} \\
 -e^{\hat{u}/2} & i\lambda^{-2}\hat{\alpha}e^{-\hat{u}/2} & 0 & -\frac{i}{\sqrt{2}}\hat{u}_{z} \\
 -ie^{\hat{u}/2} & \lambda^{-2}\hat{\alpha}e^{-\hat{u}/2} & \frac{i}{\sqrt{2}}\hat{u}_{z} & 0
\end{pmatrix}, 
\end{equation}
and
\begin{equation}\label{eq: tilde V}
\tilde{\mathcal{V}}^{\lambda} = \frac{\sqrt{2}}{2}\begin{pmatrix}
 0 & 0 & -e^{\hat{u}/2} & ie^{\hat{u}/2} \\
 0 & 0 & -i \lambda^2 \bar{\hat{\alpha}}e^{-\hat{u}/2} & \lambda^2 \bar{\hat{\alpha}}e^{-\hat{u}/2} \\
 - e^{\hat{u}/2} & - i\lambda^2 \bar{\hat{\alpha}}e^{-\hat{u}/2} & 0 & \frac{i}{\sqrt{2}}\hat{u}_{\bar{z}} \\
 ie^{\hat{u}/2} & \lambda^2 \bar{\hat{\alpha}}e^{-\hat{u}/2} & -\frac{i}{\sqrt{2}}\hat{u}_{\bar{z}} & 0
\end{pmatrix}.
\end{equation}

From the form of $\tilde \omega^{\lambda}$, it is easy to see that there exists 
a family of minimal Lagrangian surfaces $f^{\lambda}$ with fundamental 
quantities
\[
\hat u^{\lambda} = \hat u, \quad \hat \alpha^{\lambda} = \lambda^{-2} \hat \alpha, \quad 
\hat \beta^{\lambda} = \hat \beta.
\]
The induced metric of $f^{\lambda}$ is 
\[
       (e^{\hat{u}} + |\hat{\alpha}^{\lambda}|^{2}e^{-\hat{u}})\text{d}z\text{d}\bar{z} =  (e^{\hat{u}} + |\hat{\alpha}|^{2}e^{-\hat{u}})\text{d}z\text{d}\bar{z}.
\] 
Moreover, $f^{\lambda}$ is given by the sum of the first two columns of $\tilde{\mathcal F}_{\lambda}$, i.e., 
\[
f^{\lambda}=[\hat{\mathfrak f}^{\lambda}] \in Q_2^{*}, \quad 
\hat{\mathfrak f}^{\lambda} = \frac{\sqrt{2}}2(\tilde{\mathcal F}_{\lambda}^1 + i \tilde{\mathcal F}_{\lambda}^2) \in \C^4_{2},
\]
where $\tilde{\mathcal F}_{\lambda}^j (j=1, 2)$ denotes the $j$-th column of 
$\tilde{\mathcal F}_{\lambda}$.
We summarize the above discussion as the following theorem:
\begin{Theorem}\label{thm2}
   Let $f: M \to Q_{2}^{*}$ be a minimal Lagrangian immersion with induced metric $2 e^{u}\mathrm{d}z \mathrm{d}\bar z$ and holomorphic quadratic differential $\hat \alpha\,  \mathrm{d}z^{2}$. Then there exists an $\mathbb{S}^{1}$-family of minimal Lagrangian immersions $\{ f^{\lambda} \}$ with the same induced metric and holomorphic quadratic differential $\hat \alpha^{\lambda} \mathrm{d}z^{2} = \lambda^{-2} \hat \alpha \, \mathrm{d}z^{2}$.
\end{Theorem}
 \subsection{Minimal Lagrangian surfaces and spacelike maximal surfaces}
 We now compare the Maurer-Cartan form of 
a minimal Lagrangian surface with that of a spacelike maximal surface in $\mathbb{H}^{3}_{1}$. Then we establish the following theorem, which should be viewed as the anti-de Sitter counterpart of Theorem 1.9 in \cite{KZ2025}.  
\begin{Theorem}\label{Thm corresponding}
 Any spacelike maximal surface $f_{max}$ in $\mathbb H^3_{1}$ with unit normal $N$, metric 
 $2 e^{\hat u}\, \mathrm{d} z\mathrm{d}\bar{z}$ and Hopf differential $\mathcal{Q}\, \mathrm{d}z^{2}$, induces a minimal Lagrangian surface  $f=[f_{max} + iN ] \in Q_2^{*}$, whose holomorphic differential is $-i \mathcal{Q} \mathrm{d}z^{2}$ and whose metric is
 \begin{equation}\label{eq:liftmetric}
 2 e^{u}\, \mathrm{d} z\mathrm{d}\bar{z}
   = \bigl(e^{\hat u} + |\mathcal{Q}|^{2} e^{-\hat u}\bigr)\, \mathrm{d} z\mathrm{d}\bar{z}.
\end{equation}
 Conversely, given a minimal Lagrangian surface \(f\) in \(Q_2^*\) with holomorphic differential 
 $\hat \alpha \, \mathrm{d}z^2$ and metric 
 $2e^{u}\text{d}z\text{d}\bar{z}$,  after choosing a normalized horizontal lift, 
 locally if necessary, there exists a pair
\( g= (f_{\max},N)\) in \(T_1^-\mathbb H^3_1\),
 where
 \begin{equation*}
     T^{-}_{1} \mathbb H^3_{1} := \left\{ (p, v)\in T\mathbb{H}^{3}_{1} \mid v=(v_{0},v_{1},v_{2}) \in T_{p}\mathbb{H}^{3}_{1} \cong \mathbb{R}^{3}_{1},~ -v_{0}^{2} + v_{1}^{2} + v_{2}^{2} = -1  \right\}.
 \end{equation*}
 Moreover, both projections $f_{max}$ and $N$ have the same Hopf differential $i\hat \alpha  \, \mathrm{d}z^2$, and the metrics $2e^{\hat{u}}\text{d}z\text{d}\bar{z}$ of $f_{max}$ 
 and $2e^{\tilde{u}}\text{d}z\text{d}\bar{z}$ of $N$ 
 are given by 
 \begin{equation*}
  e^{\hat{u}} = e^{u} + \sqrt{e^{2u} - |\hat{\alpha}|^{2}}, \quad e^{\tilde{u}} = e^{u} - \sqrt{e^{2u} - |\hat{\alpha}|^{2}}. 
 \end{equation*}
\end{Theorem}

\begin{proof}
 Let $f_{max}$ be a spacelike maximal surface in $\mathbb{H}^{3}_{1}$ with induced metric  $ds^{2}_{M} = 2e^{\hat{u}}\text{d}z\text{d}\bar{z}$ and Hopf differential $\mathcal{Q} \mathrm{d}z^{2}$. Choose the $\SO$-frame of $f_{max}$ such that $\sigma(z_{0}) = \mathrm{Id}$ as follows: 
 \begin{equation}\label{eq: frame sigma}
     \sigma = \left( f_{max}, N , -\frac{(f_{max})_{z} + (f_{max})_{\bar{z}}}{\sqrt{2}e^{\hat{u}/2}}, \frac{i\left( (f_{max})_{z} - (f_{max})_{\bar{z}} \right)}{\sqrt{2}e^{\hat{u}/2}} \right), 
 \end{equation}
where $N$ is the normal vector of $f_{max}$. 
By setting $\mathcal{Q} \mathrm{d}z^{2} = i\hat{\alpha} \mathrm{d}z^{2}$, the Maurer-Cartan form obtained by \eqref{eq: frame sigma} is the same as \eqref{eq: Maurer-Cartan form tilde} for $\lambda = 1$. It follows that the minimal Lagrangian surface $f$ and its local lift $\hat{\mathfrak{f}}$ defined above can be obtained by 
\begin{equation}
f=[ \hat{\mathfrak{f}}] \in Q_{2}^{*}, \quad \hat{\mathfrak{f}} := \frac{\sqrt{2}}{2}\left( f_{max} + iN \right) \in \mathbb{C}^{4}_{2}, 
\end{equation}
with holomorphic differential $\langle \hat{\mathfrak{f}}_{z} , \hat{\mathfrak{f}}_{z} \rangle \mathrm{d}z^{2} = -i\mathcal{Q} \mathrm{d}z^{2}$.
 Thus the first claim follows. Moreover, note that $N$ corresponds to the same minimal Lagrangian surface $f$ in $Q_{2}^{*}$ up to a conjugation. 

Conversely, let $f$ be a minimal Lagrangian surface in $Q_{2}^{*}$ with the holomorphic differential $\hat{\alpha} \mathrm{d}z^{2}$ and $\tilde{\omega}^{\lambda}$ be the Maurer-Cartan form defined in \eqref{eq: Maurer-Cartan form tilde}. Since $\tilde{\omega}^{\lambda}|_{\lambda = 1}$ is the same as the Maurer-Cartan form of the $\SO$-frame \eqref{eq: frame sigma} of $f_{max}$, there exists a pair of spacelike maximal surfaces in $\mathbb{H}^{3}_{1}$ as follows:
\begin{equation}
    f_{max}= \frac{1}{\sqrt{2}}\big( \hat{\mathfrak{f}} + \bar{\hat{\mathfrak{f}}} \big) = \sqrt{2} 
    \operatorname{Re}  \hat{\mathfrak{f}} \in \R^4_2, \quad N = -\frac{i}{\sqrt{2}}\big( \hat{\mathfrak{f}} - \bar{\hat{\mathfrak{f}}} \big) = \sqrt{2}\operatorname{Im}\hat{\mathfrak{f}} \in \R^4_2,
\end{equation}
with $\langle f_{max}, f_{max}\rangle =\langle N, N\rangle=-1$,  and they share the same Hopf differential $\langle (f_{max})_{zz}, N  \rangle \mathrm{d}z^{2} =  i\hat{\alpha} \mathrm{d}z^{2}$. 
Similarly, there exists another pair of spacelike maximal surfaces $N = \sqrt{2} \operatorname{Re}\hat{\mathfrak{f}},~ f_{max} = \sqrt{2}\operatorname{Im}\hat{\mathfrak{f}}$ in $\mathbb{H}^{3}_{1}$. Thus \(f\) determines such a pair \((f_{\max},N)\), up to the natural
ambiguities of the horizontal lift and the choice of the normal.
This completes the proof. 
\end{proof}

\begin{Remark}
The metric defined in \eqref{eq: metric} can be thought of as half the Sasakian metric 
 on the timelike unit tangent bundle $T^{-}_{1}\mathbb H^3_{1} = \mathbb H^3_{1} \times \mathbb H^2$ for 
 a spacelike maximal surface $f_{max}$ in $\mathbb H^3_{1}$, that is, the following relation holds:
 \begin{equation*}
2 ds^2_{Lag}= 4 e^{u}\, \mathrm{d} z\mathrm{d}\bar{z}
   =  \bigl(2 e^{\hat u} + 2|\mathcal{Q}|^{2} e^{-\hat u}\bigr)\, \mathrm{d} z\mathrm{d}\bar{z}
   = ds^2_{Sasaki}.
\end{equation*} 
\end{Remark}

\section{The loop group method for minimal Lagrangian surfaces in $Q_2^*$}\label{sc:DPW}
 The PDE in \eqref{eq: sinh-Gordon equation} is merely
 the structure equation for a harmonic map from a Riemann surface into $\mathbb H^2$.
 Hence, one naturally expects a harmonic map associated with any minimal Lagrangian surface. In addition, we develop the loop group method for such surfaces, following the framework of \cite{DPW}.

\subsection{Minimal Lagrangian surfaces and harmonic maps into $\mathbb H^2$}
\label{subsc:S2harm}
To obtain an associated harmonic map to $\mathbb H^2$, we make use of the following Lie group isomorphism and Lie algebra isomorphism: 
\begin{equation}\label{eq:psi}
 \SO \cong (\SU \times \SU)/\boldsymbol{Z}_2, \quad \mathfrak{so}(2,2) \cong \mathfrak{su}(1,1) \oplus \mathfrak{su}(1,1).
\end{equation}
Here $\SU=\left\{ A \in \mathrm{SL}(2, \C) \mid \bar A^T \eta A = \eta \right\}$
with $\eta = \mathrm{diag} (1, -1)$ is the indefinite special unitary group.
The explicit form of the isomorphism can be found in \cite[Theorem 10.7, Theorem 11.10]{Yokota}.
In the loop group method, we can also use $\Lambda \mathfrak{so}(2,2) \cong \Lambda \mathfrak{su}(1,1) \oplus \Lambda \mathfrak{su}(1,1)$, where 
\[
\Lambda \mathfrak{g} := \{ g : \mathbb S^1 \to \mathfrak g \}
\]
 denotes the \textit{loop algebra} of $\mathfrak g$, which is an infinite-dimensional Banach Lie
 algebra with respect to a suitable norm \cite{PS}.
 Hence the associated Maurer-Cartan form $\hat{\omega}^{\lambda}$ in \eqref{eq: M-C hat omega} can 
  be represented as follows:
\begin{equation}
\begin{aligned}\label{eq:hatu}
\hat{U}^{\lambda}= (\hat{U}_1^{\lambda}, \hat{U}_2^{\lambda})=&\left( \begin{pmatrix}
 \frac{1}{4}\hat{u}_{z} & -\frac{i}{\sqrt{2}}\lambda^{-1} e^{\hat{u}/2} \\
  \frac{i}{\sqrt{2}}\lambda^{-1}\hat{\alpha} e^{-\hat{u}/2} & -\frac{1}{4}\hat{u}_{z}
\end{pmatrix} ,  \begin{pmatrix}
 \frac{1}{4}\hat{u}_{z} & -\frac{1}{\sqrt{2}}\lambda^{-1} e^{\hat{u}/2} \\
  \frac{1}{\sqrt{2}}\lambda^{-1}\hat{\alpha}e^{-\hat{u}/2} & -\frac{1}{4}\hat{u}_{z}
\end{pmatrix} \right), 
\end{aligned}
\end{equation}
and
\begin{equation}\label{eq:hatv}
\begin{aligned}
\hat{V}^{\lambda} =(\hat{V}_1^{\lambda}, \hat{V}_2^{\lambda})= &\left( \begin{pmatrix}
 -\frac{1}{4}\hat{u}_{\bar{z}} & -\frac{i}{\sqrt{2}}\lambda \bar{\hat{\alpha}}e^{-\hat{u}/2} \\
 \frac{i}{\sqrt{2}}\lambda e^{\hat{u}/2} & \frac{1}{4}\hat{u}_{\bar{z}}
\end{pmatrix} ,   \begin{pmatrix}
-\frac{1}{4}\hat{u}_{\bar{z}} & \frac{1}{\sqrt{2}}\lambda \bar{\hat{\alpha}}e^{-\hat{u}/2} \\
  -\frac{1}{\sqrt{2}} \lambda e^{\hat{u}/2} & \frac{1}{4}\hat{u}_{\bar{z}}
\end{pmatrix} \right).
\end{aligned}
\end{equation}
 From the form of $\left( \hat{U}_{1}^{\lambda}, \hat{U}_{2}^{\lambda} \right)$ and $\left( \hat{V}_{1}^{\lambda}, \hat{V}_{2}^{\lambda} \right)$,  the corresponding pair of maps is 
 given by $(F_{\lambda}, F_{i\lambda})$ such that 
\begin{equation}\label{eq: Maurer-Cartan form tilde12}
\hat{\omega}^{\lambda}_{1} 
= F_{\lambda}^{-1}\text{d}F_{\lambda} = \hat{U}_{1}^{\lambda} \text{d}z + \hat{V}_{1}^{\lambda} \text{d}\bar{z}, \quad 
\hat{\omega}^{\lambda}_{2} 
= F_{i\lambda}^{-1}\text{d}F_{i\lambda} = \hat{U}_{2}^{\lambda} \text{d}z + \hat{V}_{2}^{\lambda} \text{d}\bar{z}.
\end{equation}
Notice that $\hat{U}_{2}^{\lambda}$ (resp. $\hat{V}_{2}^{\lambda}$) can be obtained from $\hat{U}_{1}^{\lambda}$ (resp. $\hat{V}_{1}^{\lambda}$) by transforming $\lambda$ into $i\lambda$. From the above discussion, we obtain the following proposition.
\begin{Proposition}\label{prp:extendcorr}
Let $\hat {\mathcal F}_{\lambda}$ be the extended frame of $\hat{\mathfrak f}$. Then there exists a pair of maps $(F_{\lambda}, F_{i \lambda})$ such that 
 the Maurer-Cartan form $\hat{\omega}^{\lambda}_{1}$ of  $F_{\lambda}$ is given by \eqref{eq: Maurer-Cartan form tilde12}.  Conversely, 
 given a pair of maps $(F_{\lambda}, F_{i \lambda})$ such that the Maurer-Cartan form $\hat{\omega}^{\lambda}_{1}$ of $F_{\lambda}$ is given by \eqref{eq: Maurer-Cartan form tilde12}, there exists the extended frame $\hat {\mathcal F}_{\lambda}$ 
 of some minimal surface $\hat{\mathfrak f}$.
\end{Proposition}

Now consider the map $\check{f}: M \to \mathbb{H}^{2}$ and the Maurer-Cartan form $\hat{\omega}^{\lambda}_{1}$ of the local lift $F_{\lambda}$.
This form coincides with the Maurer-Cartan form of the extended frame of a non-conformal harmonic map from $M$ into $\mathbb{H}^2$.
Hence $F_{\lambda}$ is naturally viewed as the \emph{extended frame} of the harmonic map $\check{f}$. In addition, let $F_{\lambda}$, $F_{i\lambda}$ be the solutions of the system \eqref{eq: Maurer-Cartan form tilde12} and define 
\begin{equation}\label{eq: fmin N}
   f_{max}:= F_{\lambda}\begin{pmatrix}
 e^{-\pi i/4} & 0 \\
 0 & e^{\pi i/4}
\end{pmatrix} F_{i\lambda}^{-1}, \quad N := iF_{\lambda} \begin{pmatrix}
 e^{-\pi i /4} & 0 \\
 0 & -e^{\pi i /4}
\end{pmatrix} F_{i\lambda}^{-1}.
\end{equation}
By the results in \cite{FI2000, O2017}, $f_{max}$ is a spacelike maximal surface in $\mathbb{H}^{3}_{1}$ and normal $N$.
 Moreover, $[f_{max}+i N]$ is a minimal Lagrangian surface in $Q_2^*$ by Theorem \ref{Thm corresponding}. In this way all minimal Lagrangian surfaces can be obtained by a harmonic map into $\mathbb H^2$.
\subsection{The DPW method }\label{Subsection the DPW method}\label{sbsc:DPW}
 Now we apply the generalized Weierstrass type representation (the DPW method)
for harmonic maps in $\mathbb H^2$, see basic construction in \cite{BRS2010, O2017}. Note that the extended frame $F_{\lambda}$ takes values in 
\[
\Lambda G_{\sigma} = \{g : \mathbb S^1  \to G \mid  \sigma g (\lambda) = g(-\lambda) \},
\]
where $G = \SU$ and $\sigma (g) = \operatorname{Ad} \operatorname{diag}(1, -1)  (g)$ is the involution associated with the symmetric space $\mathbb H^2 = \SU/\Uone$. With respect to a suitable Banach 
topology, $\Lambda G_{\sigma}$ is a Banach Lie group (the \textit{loop group} of 
$G$, see \cite{PS}). The \textit{Birkhoff decomposition} (\cite{PS, DPW}) of $F_{\lambda}$
\[
 F_{\lambda} = F_{-} F_{+}\quad \mbox{with $F_- \in \Lambda_*^- G^{\mathbb C}_{\sigma}$
 and $F_+\in \Lambda^+ G^{\mathbb C}_{\sigma}$}
\]
yields the meromorphic dependence of $F_{-}$ \cite[Lemma 2.6]{DPW}. 
Here $G^{\mathbb C} = \SL$ and $\Lambda^{\pm} G^{\mathbb C}_{\sigma}$ 
denote the subgroups of $\Lambda G^{\mathbb C}_{\sigma}$ which can be 
extended to the inside (resp. outside) of the unit disk in $\mathbb{C} P^1$ 
when the plus sign (resp. the minus sign) is chosen. The subscript $*$ denotes 
the identity normalization at $\lambda =0$.

Conversely, minimal Lagrangian surfaces in the complex hyperbolic quadric $Q_{2}^{*}$ can be constructed in the following four steps:
\begin{enumerate}
    \item[$\mathbf{1.}$] Solve the initial-value problem: 
    \begin{equation}
        \mathrm{d}\Phi = \Phi \xi, \quad \Phi(z_{0}) = \Phi_{0} \in \Lambda \SL_{\sigma},
    \end{equation}
    to obtain a unique map $\Phi : \mathbb D  \to \Lambda \SL_{\sigma}$.
    
    \item[$\mathbf{2.}$] Compute the \textit{Iwasawa decomposition} (see \cite{PS, BRS2010})  
    of $\Phi$ pointwise on $\mathbb D$:
    \begin{equation}\label{eq: F lambda}
        \Phi = F_{\lambda} B, \quad F_{\lambda} \in \Lambda \SU_{\sigma}, \quad B \in \Lambda^{+} \SL_{\sigma},
    \end{equation}
    Then, by \cite[Remark 3.4]{BRS2010}, $F_{\lambda}$ is the extended frame of a harmonic map into $\mathbb H^2$. Note that the Iwasawa decomposition is not global in general. So the decomposition here is just local around the base point $z_0 \in \mathbb D$, see 
    \cite{BRS2010}. Set the pair of maps given by another map $F_{i\lambda}$ as
    \begin{equation}
        \left( F_{\lambda}, F_{i\lambda} \right) \in \Lambda \SU_{\sigma}\times \Lambda \SU_{\sigma}.
    \end{equation}
    
    \item[$\mathbf{3.}$] 
    Using the loop group isomorphism 
    \begin{equation}\label{eq:loopisom}
        \Lambda \SO_{\sigma} \cong (\Lambda \SU_{\sigma} \times \Lambda \SU_{\sigma}) /\boldsymbol{Z}_2, 
    \end{equation}
    together with Proposition \ref{prp:extendcorr}, one obtains 
    the extended frame $\mathcal{F}_{\lambda} \in \Lambda \SO_{\sigma} $ of 
    some minimal Lagrangian immersion into $Q_2^{*}$. 
    
    \item[$\mathbf{4.}$] Finally, by using Proposition \ref{Prop} below, we obtain 
 a family of minimal Lagrangian immersions $f^{\lambda}$ into $Q_2^{*}$.
\end{enumerate}
\begin{Remark}
Although the involution $\sigma = \operatorname{Ad} \operatorname{diag} (1, 1, -1, -1)$ of $\SO$, which defines the twisted loop group $\Lambda \SO_{\sigma}$, differs from the involution $\operatorname{Ad} \operatorname{diag} (1, -1)$ of $\SU$ defining $\Lambda \SU_{\sigma}$, the homomorphism \eqref{eq:psi} extends naturally to the level of twisted loop groups as in \eqref{eq:loopisom}. The corresponding Maurer-Cartan forms are given in \eqref{eq:hatu} and \eqref{eq:hatv}.
\end{Remark}
In the following proposition and corollary, we will make use of the Pauli matrices
\begin{equation*}
    \sigma_{1} =  \begin{pmatrix} 0 & 1 \\
 1 & 0
\end{pmatrix},\quad \sigma_{2} = \begin{pmatrix}
 0 & -i \\
 i & 0
\end{pmatrix},\quad \sigma_{3} = \begin{pmatrix} 1 & 0 \\
 0 & -1
\end{pmatrix}. 
\end{equation*} 
\begin{Proposition}\label{Prop}
Let $F_{\lambda}$ be the extended frame defined above.
 Set 
 \begin{equation}\label{eq:Xlambda}
     X^{\lambda} = (X^{\lambda}_{ij}) := F_{\lambda} F^{-1}_{i\lambda}, \quad Y^{\lambda} = (Y^{\lambda}_{ij}) := i F_{\lambda} \sigma_{3} F_{i\lambda}^{-1}.
 \end{equation}
 Then the associated family $\{f^{\lambda} \}$ of a minimal Lagrangian surface in $Q_2^*$ 
 can be represented by 
     \begin{equation*}
     f^{\lambda}=
\left[\left(
\mathrm{Re} (X^{\lambda}_{11}) + i\mathrm{Re} (Y^{\lambda}_{11}),  
\mathrm{Im} (X^{\lambda}_{11}) + i\mathrm{Im} (Y^{\lambda}_{11}), 
\mathrm{Re} (X^{\lambda}_{21}) + i\mathrm{Re} (Y^{\lambda}_{21}), 
\mathrm{Im} (X^{\lambda}_{21}) + i\mathrm{Im} (Y^{\lambda}_{21})
\right)
\right].
\end{equation*}
where $\operatorname{Re}$ and $\operatorname{Im}$ denote the real and the imaginary part, respectively.
\end{Proposition}
\begin{proof}
By the Lie group isomorphism given in \eqref{eq:psi} and a straightforward computation, we obtain the result.
\end{proof}
Since $Q_2^*$ is isomorphic to $\mathbb H^2 \times \mathbb H^2$, we 
have a representation formula in $\mathbb H^2 \times \mathbb H^2$ 
by the extended frame $F_{\lambda}$.
\begin{Corollary}\label{coro:formula}
Let $F_{\lambda}$ be the extended frame defined above.
Define a map 
\begin{equation}
\Phi_{\lambda}= 
(\phi_{\lambda}, \psi_{\lambda})
: \mathbb D \to  \mathbb H^2 \times \mathbb H^2
\end{equation}
by $(\phi_{\lambda}, \psi_{\lambda})
:= (iF_{\lambda} \sigma_3 F_{\lambda}^{-1}, iF_{i\lambda} \sigma_3 F_{i\lambda}^{-1})$.
Then  $\{\Phi_{\lambda}\}_{\lambda \in \mathbb S^1}$ 
is a family of minimal Lagrangian surfaces.
\end{Corollary}
\begin{proof}
By the definition of the maps $\phi_{\lambda}$ and $\psi_{\lambda}$, 
it is easy to check $\phi_{\lambda}, \psi_{\lambda} \in \mathfrak{su}(1,1) \cong \mathbb{R}^{3}_{1}$ for each $\lambda \in \mathbb{S}^{1}$. Under this identification, we have $\la \phi_{\lambda}, \phi_{\lambda}  \ra = -(1/2) \mathrm{trace}( \phi_{\lambda} \sigma_{2} \phi_{\lambda}^{T} \sigma_{2} ) = -1$, 
and similarly for $\psi_{\lambda}$, i.e. $\phi_{\lambda}, \psi_{\lambda} \in \mathbb{H}^{2}$. Moreover, a straightforward computation shows that 
\begin{equation}\label{eq: metric H2H2}
    \la (\phi_{\lambda})_{\bar z},  (\phi_{\lambda})_{z} \ra = 
 \la (\psi_{\lambda})_{\bar z},  (\psi_{\lambda})_{z} \ra = 2 e^{u}, \quad
        \la (\Phi_{\lambda})_{\bar z}, (\Phi_{\lambda})_{z}\ra = 4e^{u}, \quad
        \left\langle (\Phi_{\lambda})_{z}, (\Phi_{\lambda})_{z}  \right \rangle = 0    
\end{equation} 
hold. That is, $\Phi_{\lambda}$ is a conformal immersion, and moreover,  it is Lagrangian.
Another straightforward computation shows that $(\phi_{\lambda})_{z\bar{z}} = 2e^{u} \phi_{\lambda}$ and   $(\psi_{\lambda})_{z\bar{z}} = 2e^{u} \psi_{\lambda}$ hold.
Thus $(\Phi_{\lambda})_{z \bar z} = 2e^{u} \Phi_{\lambda}$ holds, and $\Phi_{\lambda}$ is harmonic, and 
it is minimal.\footnote{This can also be seen from \eqref{eq: GVWX 1} in Appendix \ref{app: CU}.}
\end{proof}

\section{Examples through the DPW method}\label{sc:Ex}
By using the DPW method introduced in Section \ref{sc:DPW}, we will construct the following types of minimal Lagrangian surfaces in $Q_2^{*}$. Again note that the Iwasawa decomposition is not global in general. So the examples given here are local around some base point $z_0 \in \mathbb D$.
\subsection{Basic examples}
The basic examples are the open part of the diagonal surface and the product of geodesics, see also \cite{GVWX}. 
\subsubsection{Open part of the diagonal surface}
Define
\begin{equation*}
    \xi := \lambda^{-1} \begin{pmatrix}
 0 & 1\\
 0 & 0
\end{pmatrix}\text{d}z
\end{equation*}
for $z \in \mathbb{C}$. 
The solution of \(d\Phi=\Phi\xi\) with \(\Phi(0)=\id\) is
$\Phi=\exp(z\,\xi/dz)$.
 Moreover,  the Iwasawa decomposition of $\Phi = F_{\lambda} B$ is given by 
\begin{equation*}
    F_{\lambda} = \frac{1}{\sqrt{1 - |z|^{2}}}\begin{pmatrix}
 1 & z\lambda^{-1} \\
 \bar{z}\lambda & 1
\end{pmatrix}. 
\end{equation*}
By Proposition \ref{Prop}, we obtain a family of open parts of the diagonal surface $\{ f^{\lambda} \}$ parameterized by $\lambda \in \mathbb{S}^{1}$: 
\begin{equation*}
     f^{\lambda}=
\left[\left(
1 - i|z|^{2},~  
-|z|^{2} + i,~ 
z\lambda^{-1} - i\bar{z}\lambda,~ 
-\bar{z}\lambda + iz\lambda^{-1}
\right)
\right].
\end{equation*}

\subsubsection{Product of geodesics}
Define
\begin{equation*}
    \xi := \lambda^{-1} \begin{pmatrix}
 0 & 1\\
 1 & 0
\end{pmatrix}\text{d}z
\end{equation*}
for $z \in \mathbb{C}$. It is easy to solve the ODE $\mathrm{d} \Phi = \Phi \xi$ by $\Phi = \exp ( z \xi/ \text{d}z)$ with $\Phi (0)=\id$.
 Moreover,  the Iwasawa decomposition of $\Phi = F_{\lambda} B$ is given by 
\begin{equation*}
    F_{\lambda} = \begin{pmatrix}
 \cosh (\lambda^{-1}z + \bar{z}\lambda) & \sinh (\lambda^{-1}z + \bar{z}\lambda)\\
 \sinh (\lambda^{-1}z + \bar{z}\lambda) & \cosh (\lambda^{-1}z + \bar{z}\lambda)
\end{pmatrix}.
\end{equation*}
By Proposition \ref{Prop}, we obtain a family of products of geodesics $\{ f^{\lambda} \}$ parameterized by $\lambda \in \mathbb{S}^{1}$: 
\begin{equation*}
     f^{\lambda}=
\left[\left(
\cos  s ,~  
i\cos  t ,~ 
i\sin  s ,~ 
\sin  t 
\right)
\right],
\end{equation*}
where
\begin{equation*}
    s= \lambda^{-1}z - \bar{z}\lambda - i(\lambda^{-1}z + \bar{z}\lambda), \quad t= \lambda^{-1}z - \bar{z}\lambda + i(\lambda^{-1}z + \bar{z}\lambda).
\end{equation*}

\subsection{Equivariant and radially symmetric examples}\label{sbsc:Equ}
 We now show two new examples of minimal Lagrangian surfaces in $Q_2^{*}$.
\subsubsection{$\mathbb R$-equivariant minimal Lagrangian surfaces}\label{subsub:equiv}
\begin{Definition}[$\mathbb R$-equivariant potentials, \cite{BRS2010}]
 Define 
 \begin{equation}\label{eq: Delaunay potential}
    \xi = A(\lambda) \, \mathrm{d} z, \quad \text{where} \quad A(\lambda) = \begin{pmatrix}
 c & a\lambda^{-1} + b\lambda \\
 -a\lambda - b\lambda^{-1} & -c
\end{pmatrix}, 
\end{equation}
with $a, b \in \R^{*}$ and $c \in \R$.  
We call such potentials the \textit{equivariant potentials}.
\end{Definition}
It is easy to see that  $\Phi = \exp ( z \cdot A)$ is the unique solution of $\mathrm{d} \Phi = \Phi \xi$ 
 with the initial condition  $\Phi(0) = \id$. Let $\Phi = F_{\lambda} B $ be the Iwasawa decomposition of $\Phi$ (see below for the explicit form of $F_{\lambda}$). 
 As discussed in \cite[Section 5.1]{BRS2010}, the transformation in the imaginary part $z \to z + i\theta$ leads to
\begin{equation}
    \gamma_{\theta} : z = x + iy \longmapsto x + i(y + \theta).
\end{equation}
Then the following transformation rule of $F_{\lambda}$ follows: 
\begin{equation}\label{eq: gamma Flambda}
    \gamma^{*}_{\theta} F_{\lambda} =  F_{\lambda}\left( z+i \theta , \, \bar{z}-i\theta , \, \lambda \right) = \exp \left( i\theta A (\lambda) \right) \cdot F_{\lambda}\left( z ,\, \bar{z} ,\, \lambda \right).
\end{equation}
 Note that $i \theta A (\lambda)$ takes values in $\Lambda \mathfrak{su} (1,1)_{\sigma}$ and thus 
 $\exp \left( i\theta A (\lambda) \right)$ takes values in $\Lambda \mathrm{SU} (1,1)_{\sigma}$.

  The general definition of an equivariant surface can be found in \cite{BRS2010}, then a straightforward computation shows that the minimal Lagrangian surface constructed 
by the equivariant potential $\xi$ in \eqref{eq: Delaunay potential} is an equivariant surface.
\begin{Proposition}\label{Prop: Requivariant}
    Let $\xi$ be an $\mathbb R$-equivariant potential defined in \eqref{eq: Delaunay potential}
    and let $F_{\lambda} \in \Lambda \SU_{\sigma}$ for $\lambda \in \mathbb{S}^{1}$ be the corresponding
    extended frame.
    Then the minimal Lagrangian surface $f^{\lambda} : M \to Q_{2}^{*}$ constructed by $\left( F_{\lambda }, F_{i\lambda} \right)$ is equivariant, that is,
    \[
     \hat {\mathfrak f}^{\lambda}\left( z + i\theta , \, \bar{z} - i\theta , \, \lambda \right)=
     \psi\big(
      ( 
     \exp (i\theta A(\lambda)) , \exp(-i\theta A(i\lambda)) 
      )
      \big) \, \hat{\mathfrak f}^{\lambda}(z, \bar{z}, \lambda) 
    \]
    holds, where $\hat {\mathfrak f}^{\lambda}$ is the horizontal lift of $f^{\lambda}$ and 
    $\psi : \Lambda \SU_\sigma \times \Lambda \SU_\sigma \to \Lambda \SO_\sigma$ is the loop group homomorphism.
\end{Proposition}

By Theorem 5.1 of \cite{BRS2010}, we obtain the explicit form of $F_{\lambda}$ as
\begin{equation*}
   F_{\lambda} = \begin{pmatrix}
 \frac{4ab\lambda^{2}+v^{2}}{t\sqrt{2v(a\lambda^{2} + b)(4ab \lambda^{2} + v^{2})}} \left( t\cosh \hat{t} + c\lambda \sqrt{-t}\sinh \hat{t} \right) & \frac{-\lambda \sqrt{-t} (2cv + v')\cosh \hat{t} + (-2tv + c\lambda^{2}v')\sinh \hat{t}}{\sqrt{-2t}\sqrt{v(a\lambda^{2} + b)(4ab\lambda^{2} + v^{2})}} \\
 \frac{\sqrt{v(a\lambda^{2} + b)(4ab \lambda^{2} + v^{2})}}{v\sqrt{-2t}} \sinh \hat{t} & \frac{v(a\lambda^{2}+b)}{\sqrt{2v(a\lambda^{2} + b)(4ab \lambda^{2} + v^{2})}}\left( 2\cosh \hat{t} - \frac{\lambda v'}{v\sqrt{-t}}\sinh \hat{t} \right)
\end{pmatrix},
\end{equation*}
where $\mathbf{f}$, $t$ and $\hat{t}$ are given by 
\begin{equation*}
    \mathbf{f}(x) = \int^{x}_{0} \frac{2\text{d} s}{1 + (4ab\lambda^2)^{-1}v^{2}(s)}, \quad
    t = ab + (a^{2} + b^{2} - c^{2})\lambda^{2} + ab \lambda^{4}, \quad \hat{t} = \sqrt{-t}~(\mathbf{f} - z)\lambda^{-1}.
\end{equation*}
We choose $\kappa_{1},\kappa_{2}$ so that $x \in (-\kappa_{1}^{2}, \kappa_{2}^{2})$ is the largest interval for which a solution $v = v(x)$ of 
\begin{equation*}
    (v')^{2} = (v^{2} - 4a^{2})(v^{2} - 4b^{2}) + 4c^{2}v^{2}, \quad  v''=2v(v^{2}-2a^{2}-2b^{2}+2c^{2}), \quad v(0) = 2b,
\end{equation*}
is finite and never zero ($'$ denotes $\frac{\mathrm{d}}{\mathrm{d}x}$). When $c\ne 0$, we require $v'(0)$ and $-bc$ to have the same sign. 
    Furthermore, one can write the solution $v(x)$ in the form 
    \begin{equation}\label{eq: solution v}
    v(x) = \sqrt{Z_{1}} \operatorname{sn}\left(\sqrt{Z_{2}}(x - x_{0}), \sqrt{Z_{1}/Z_{2}} \right) \quad \text{with}~ x \in (-\kappa_{1}^{2}, \kappa_{2}^{2}),
\end{equation}
where $\sqrt{Z_{1}/Z_{2}}$ is the modulus, $x_{0}$ can be obtained by initial condition $v(0) = 2b$ and
\begin{align*}
    Z_{1} &= 2(a^{2} + b^{2} - c^{2}) - 2\sqrt{(a^{2} + b^{2} - c^{2})^{2} - 4a^{2}b^{2}}, \\
    Z_{2} &= 2(a^{2} + b^{2} - c^{2}) + 2\sqrt{(a^{2} + b^{2} - c^{2})^{2} - 4a^{2}b^{2}}. 
\end{align*}
It is also easy to compute the map $F_{i\lambda}$. Then by Proposition \ref{Prop}, we obtain the explicit form of this equivariant surface in $Q_{2}^{*}$. We now give the sufficient conditions such that $f^{\lambda}$ is well-defined \red{on the quotient of $\Sigma := \{ z = x+iy\in \mathbb{C} | -\kappa_{1}^{2} < x < \kappa_{2}^{2} \}$ by the period \(z\mapsto z+2\pi i\).} 
In this paper, we call a complete equivariant minimal Lagrangian annulus
in \(\mathbb H^2\times \mathbb H^2\) \textit{catenoid-type} if it has two complete annular ends and
its two projections to the \(\mathbb H^2\)-factors are rotationally symmetric.
\begin{Theorem}[Catenoid-type examples]\label{prp:closing}
 Let $f^{\lambda}$ be the \(\mathbb R\)-equivariant minimal Lagrangian surface given in 
 {\rm Proposition \ref{Prop: Requivariant}}. Then \(f^{\lambda}|_{\lambda=\lambda_0}\), where \(\lambda_0\in \mathbb S^1\),
 is well-defined on the quotient of \(\Sigma\) by the period
 \(z\mapsto z+2\pi i\) if the eigenvalues of \(A(\lambda_0)\) and
\(A(i\lambda_0)\) are non-zero half-integers (in particular, non-zero
integers). Equivalently, the closing condition is
\begin{equation}\label{eq: well define condition}
    2 \sqrt{c^{2} - |\lambda_0 a + \lambda_0^{-1}b|^{2}}, \quad
    2 \sqrt{ c^{2} - |\lambda_0 a - \lambda_0^{-1}b|^{2}}
    \in  \mathbb Z_{> 0}.
\end{equation}
Moreover, such \(f^{\lambda}|_{\lambda=\lambda_0}\) descends to a complete
catenoid-type annulus with two complete annular ends.
\end{Theorem}

\begin{proof}
Set
\[
m=
2 \sqrt{c^{2} - |\lambda_0 a + \lambda_0^{-1}b|^{2}},
\qquad
n=
2 \sqrt{c^{2} - |\lambda_0 a - \lambda_0^{-1}b|^{2}}.
\]
A direct computation of the characteristic polynomial of \(A(\lambda)\) in
\eqref{eq: Delaunay potential} shows that the eigenvalues of \(A(\lambda_0)\)
and \(A(i\lambda_0)\) are
\[
\pm \frac{m}{2},\qquad \pm \frac{n}{2},
\]
respectively. Hence \(\pm m/2\) and \(\pm n/2\) are non-zero half-integers
(integers or genuine half-integers) if and only if
\(m,n\in\mathbb Z_{>0}\), which is precisely
\eqref{eq: well define condition}.

By Proposition \ref{Prop: Requivariant}, the action of the period
\(z\mapsto z+2\pi i\) is governed by the monodromy matrices
\[
\exp(2\pi i A(\lambda_0)),\qquad
\exp(2\pi i A(i\lambda_0)).
\]
Since the eigenvalues of \(A(\lambda_0)\) and \(A(i\lambda_0)\) are
\(\pm m/2\) and \(\pm n/2\) with \(m,n\in\mathbb Z_{>0}\), we have
\[
\exp(2\pi i A(\lambda_0))=(-1)^m\id,
\qquad
\exp(2\pi i A(i\lambda_0))=(-1)^n\id ,
\]
where we used \(\exp(2\pi i\cdot(\pm m/2))=e^{\pm\pi i m}=(-1)^m\) (and
analogously for \(n\)). Thus, if \(m\equiv n \pmod 2\), the horizontal lift is
preserved by the period, whereas if \(m\not\equiv n \pmod 2\), the horizontal
lift changes sign. In both cases the projective point in \(Q_2^*\) is
unchanged, and hence \(f^{\lambda}|_{\lambda=\lambda_0}\) descends to the
quotient of \(\Sigma\) by \(z\mapsto z+2\pi i\). This proves the first claim.

We next explain why the descended surface is of catenoid type. The conditions
\eqref{eq: well define condition} imply
\[
|\lambda_0 a+\lambda_0^{-1}b|^{2}-c^{2}<0,
\qquad
|\lambda_0 a-\lambda_0^{-1}b|^{2}-c^{2}<0.
\]
By Corollary 5.3 in \cite{BRS2010}, the corresponding spacelike rotational
CMC surfaces in \(\mathbb R^3_1\) associated with the two factors have timelike
axes. Therefore, for the well-defined \(\mathbb R\)-equivariant minimal
Lagrangian surface \(f^{\lambda}|_{\lambda=\lambda_0}\) in
\(\mathbb H^2\times \mathbb H^2\), the two projections onto the
\(\mathbb H^2\)-factors are precisely the Gauss maps of spacelike rotational
CMC surfaces with timelike axes.

By \eqref{eq: gamma Flambda}, we have
\begin{equation}\label{eq: gamma Flambda closing}
    \gamma^{*}_{\theta} F_{\lambda}|_{\lambda = \lambda_{0}}
    =
    \exp(i\theta A|_{\lambda = \lambda_{0}})\,
    F_{\lambda}|_{\lambda = \lambda_{0}},
\end{equation}
where \(A=A(\lambda)\) is defined in \eqref{eq: Delaunay potential}. Since
there exist \(P_1,P_2\in \SU\) such that
\[
    P_{1}^{-1} A|_{\lambda = \lambda_{0}} P_{1}
    =
    \operatorname{diag}\left( -\frac{m}{2}, \frac{m}{2} \right),
    \qquad
    P_{2}^{-1} A|_{\lambda = i\lambda_{0}} P_{2}
    =
    \operatorname{diag}\left( -\frac{n}{2}, \frac{n}{2} \right),
\]
the map
\[
    \Phi_{\lambda}=(\phi_{\lambda},\psi_{\lambda})
    : \mathbb D\to \mathbb H^2\times \mathbb H^2
\]
defined in Corollary \ref{coro:formula} gives rise to
\begin{equation}\label{eq: Phi hat closing}
    \hat{\Phi}_{\lambda}|_{\lambda = \lambda_{0}}
    :=
    \left( \hat{\phi}_{\lambda}, \hat{\psi}_{\lambda} \right)
    \big|_{\lambda = \lambda_{0}}
    =
    \left(
    \operatorname{Ad}(P_{1}^{-1})\phi_{\lambda},
    \operatorname{Ad}(P_{2}^{-1})\psi_{\lambda}
    \right)
    \big|_{\lambda = \lambda_{0}} .
\end{equation}
Then
\begin{equation}\label{eq: gamma Phi}
    \gamma^{*}_{\theta} \hat{\Phi}_{\lambda}|_{\lambda = \lambda_{0}}
    =
    \left(
    \operatorname{Ad}
    \left(
    \operatorname{diag}(e^{-mi\theta/2}, e^{mi\theta/2})
    \right)
    \hat{\phi}_{\lambda}|_{\lambda = \lambda_{0}},
    \operatorname{Ad}
    \left(
    \operatorname{diag}(e^{-ni\theta/2}, e^{ni\theta/2})
    \right)
    \hat{\psi}_{\lambda}|_{\lambda = \lambda_{0}}
    \right).
\end{equation}
Represent
\[
   \hat{\phi}_{\lambda_{0}}
   =
   \begin{pmatrix}
      ix_{1} & x_{2} + ix_{3} \\
      x_{2} - ix_{3} & -ix_{1}
   \end{pmatrix},
   \qquad
   \hat{\psi}_{\lambda_{0}}
   =
   \begin{pmatrix}
      iy_{1} & y_{2} + iy_{3} \\
      y_{2} - iy_{3} & -iy_{1}
   \end{pmatrix}
\]
as points of
\(\mathbb H^2\subset \mathbb R^3_1\cong \mathfrak{su}(1,1)\), and project
them onto the Poincar\'e disk \(\mathbb D^2\subset \mathbb C\) by
\begin{equation}\label{eq: w1w2}
    w_{1}:=\frac{x_{2}+ix_{3}}{1+x_{1}},
    \qquad
    w_{2}:=\frac{y_{2}+iy_{3}}{1+y_{1}}.
\end{equation}
By \eqref{eq: gamma Phi} and \eqref{eq: w1w2}, we obtain
\[
    \gamma^{*}_{\theta}w_{1}=e^{-mi\theta}w_{1},
    \qquad
    \gamma^{*}_{\theta}w_{2}=e^{-ni\theta}w_{2}.
\]
Thus the \(y\)-translation gives rotations of the profile curves
\(w_1|_{y=0}\) and \(w_2|_{y=0}\). Hence the two projections of the resulting surface to the $\mathbb{H}^{2}$-factors are rotationally symmetric.

It remains to prove the completeness of the induced metric. By
\eqref{eq: metric}, Corollary 5.8 in \cite{BRS2010}, and the generalized
Lawson correspondence \cite{P1990,AA1998}, the metric is
\red{
\begin{equation}\label{eq: catenoid metric}
  ds^{2}_{M}
  =
  2e^{u}\, dz\,d\bar z
  =
  \frac12
  \left(v(x)^{2}+16|ab|^{2}\,v(x)^{-2}\right)dz\,d\bar z .
\end{equation} }
The closing condition makes the surface descend to
\[
    \Sigma/(z\sim z+2\pi i)
    =
    (-\kappa_1^2,\kappa_2^2)\times \mathbb R/2\pi\mathbb Z,
\]
which is an annulus. Its two ends correspond to
\[
    x\to -\kappa_1^2,
    \qquad
    x\to \kappa_2^2 .
\]

We show that these two annular ends are complete. Since \(a,b\ne 0\), the
coefficient of the metric in \eqref{eq: catenoid metric} is strictly positive.
Moreover, by the inequality \red{ \(s^{2}+16|ab|^{2}s^{-2}\ge \left(|s|+4|ab|\,|s|^{-1}\right)^{2}/2\) } (for instance), the length of any curve tending to an end is
bounded from below by a positive constant multiple of \red{
\[
    \int
    \sqrt{v(x)^2+16|ab|^{2}\,v(x)^{-2}}\,|dx|
    \;\ge\;
    \mathrm{const}\cdot
    \int
    \left(|v(x)|+|v(x)|^{-1}\right)|dx|.
\] }
It is therefore enough to show that
\[
    \int
    \left(|v(x)|+|v(x)|^{-1}\right)|dx|
\]
diverges near each endpoint.

Recall that \(v\) is defined on the maximal interval
\(
    (-\kappa_1^2,\kappa_2^2)
\)
on which it is finite and non-zero, and satisfies
\begin{equation}\label{eq: v ODE}
    (v')^2
    =
    (v^2-4a^2)(v^2-4b^2)+4c^2v^2 .
\end{equation}
Let \(x_0\) be a finite endpoint of this maximal interval. If \(v(x)\) had a
finite non-zero limit as \(x\to x_0\), the ordinary differential equation
\eqref{eq: v ODE} would extend \(v\) beyond \(x_0\), contradicting the
maximality of the interval. Thus, at a finite endpoint, either \(v(x)\to 0\)
or \(|v(x)|\to\infty\).

\smallskip
\textrm{Case 1: \(v(x)\to 0\).}
 Then \eqref{eq: v ODE} gives
\[
    (v')^2\to 16a^2b^2>0.
\]
In particular, \(v'\) is bounded near \(x_0\). Extending \(v\) continuously
by \(v(x_0)=0\), the mean value theorem gives
\[
    |v(x)|\le C|x-x_0|
\]
for \(x\) sufficiently close to \(x_0\). Hence
\[
    \int |v(x)|^{-1}\,dx
    \ge
    C^{-1}\int \frac{dx}{|x-x_0|}
    =
    \infty .
\]

\smallskip
\textrm{Case 2: \(|v(x)|\to\infty\).}
Expanding \eqref{eq: v ODE}, the dominant term as \(|v|\to\infty\) is
\(v^{4}\), so
\[
    (v')^{2}
    = v^{4} - 4(a^{2}+b^{2}-c^{2})\,v^{2} + 16a^{2}b^{2}
    \;\le\; C^{2}\,v^{4}
\]
for some positive constant \(C\) and all \(|v|\) sufficiently large; hence
\(|v'|\le C|v|^{2}\). Restricting to a sufficiently small end-neighborhood on
which \(|v|\) is monotone, we have
\[
    d|v|\le |v'|\,|dx|\le C|v|^{2}\,|dx|,
\]
so
\(
    |dx|\ge C^{-1}|v|^{-2}\,d|v|,
\)
and therefore
\[
    \int |v(x)|\,|dx|
    \;\ge\;
    C^{-1}\int^{\infty}\frac{d|v|}{|v|}
    \;=\;\infty .
\]

\smallskip 
\textrm{Case 3: an endpoint at infinity.}
If one endpoint of the interval is at infinity, then by AM--GM, we have \red{ 
\begin{equation*}
    v(x)^{2}+16|ab|^{2}\,v(x)^{-2}
    \;\ge\; 8|ab|>0,
\end{equation*} }
and hence the length to the endpoint is again infinite.

In all three cases the length to each endpoint diverges. Therefore both
annular ends are complete, and the surface is a complete catenoid-type annulus
with two complete annular ends.
\end{proof}

\begin{Remark}
We give an example of Theorem~\ref{prp:closing}. Let $n = 8, m = 4, \lambda_{0} = \sqrt{3}/2 + i/2, a = 1, b = 6, c=-\sqrt{47}$.  Then the solution $v(x)$ has the form \eqref{eq: solution v}, where $Z_{1} = -20 - 4i \sqrt{11},~ Z_{2} = -20 + 4i \sqrt{11}$. When $x \in (-\kappa^{2}_{1}, \kappa_{2}^{2})$, $y = 0.01$,  the profile curves of $\hat{\phi}_{\lambda_{0}}$ and $\hat{\psi}_{\lambda_{0}}$ projected onto $\mathbb{D}^{2}$ are shown in Figure \ref{figure2}. 
\end{Remark}

\begin{figure}[htbp]
\centering

\begin{subfigure}{0.4\textwidth}
\centering
\includegraphics[width=0.9\linewidth]{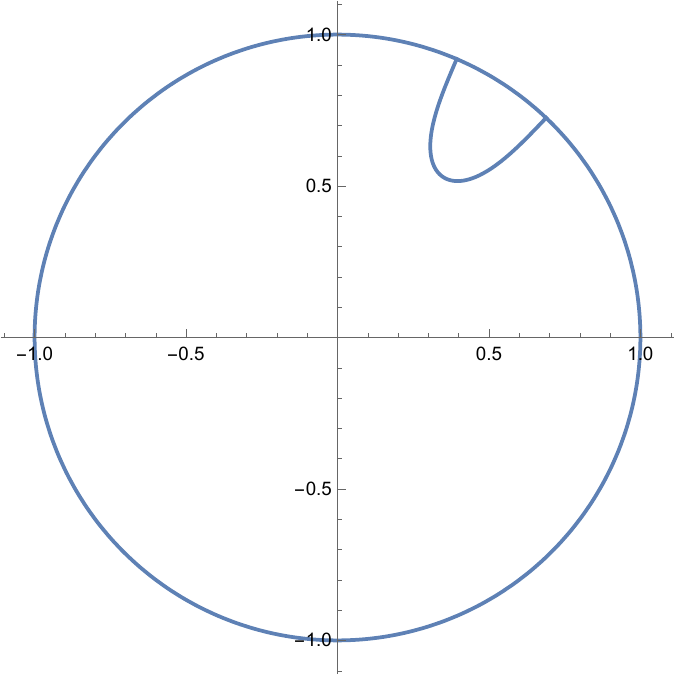}
\caption{The projection of the profile curve  of $\hat{\phi}_{\lambda_{0}}$ onto $\mathbb{D}^{2}$}
\end{subfigure}
\hfill
\begin{subfigure}{0.4\textwidth}
\centering
\includegraphics[width=0.9\linewidth]{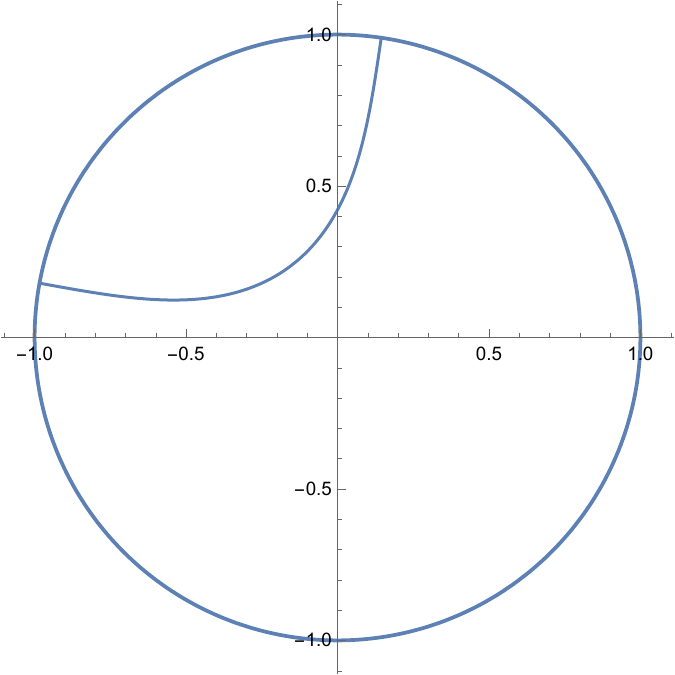}
\caption{The projection of the  profile curve of $\hat{\psi}_{\lambda_{0}}$ onto $\mathbb{D}^{2}$}
\end{subfigure}

\caption{For $x \in (-\kappa^{2}_{1}, \kappa_{2}^{2}),~ y = 0.01$, the projections of the profile 
curves of $\hat{\phi}_{\lambda_{0}}$ and $\hat{\psi}_{\lambda_{0}}$ onto $\mathbb{D}^{2}$.}\label{figure2}
\end{figure}

\subsubsection{Radially symmetric minimal Lagrangian surfaces}
\begin{Definition}[Radially symmetric potentials, \cite{BRS2010}]
Define 
\begin{equation}\label{eq: Smyth potential}
    \xi = \lambda^{-1} \begin{pmatrix}
 0 & 1 \\
 cz^{k} & 0
\end{pmatrix} \text{d} z, 
\end{equation}
for $z \in \Sigma = \mathbb{C}$ and $ k \in \mathbb N$ and some $c \in \mathbb{C} \setminus (\mathbb{S}^{1} \cup \{ 0 \} ) $. 
Here we call such potentials the \textit{radially symmetric potentials}. 
\end{Definition}

Let $R_{\ell}(z) = e^{2\pi i \ell/(k+2)} \bar{z}$ be the reflections 
of the domain $\mathbb{C}$, for $\ell \in \{ 0 , 1, \dots, k+1 \}$. Note that
\begin{equation}\label{eq:symxi}
    \xi (R_{\ell}(z), \lambda) = A_{\ell} \xi(\bar{z} , \lambda) A_{\ell}^{-1}, \quad \text{with} \quad  A_{\ell} = \begin{pmatrix}
 e^{\frac{\pi i \ell}{k+2}} & 0 \\
 0 & e^{\frac{-\pi i \ell}{k+2}}
\end{pmatrix} \in \SU
\end{equation}
holds.
Let $\Phi$ be the solution of $d \Phi = \Phi \xi$ with $\Phi(z_0) = \id$ and 
consider the Iwasawa decomposition $\Phi = F_{\lambda} B$. For $c \in \mathbb{C} \setminus (\mathbb{S}^{1} \cup \{ 0 \} )$, the Iwasawa decomposition of $\Phi$ cannot be carried out explicitly, and hence an explicit description of all radially symmetric surfaces cannot be obtained. By \eqref{eq:symxi}, we have
\begin{equation*}
    F(R_{\ell}(z),\overline{R_{\ell}(z)}, \lambda) = A_{\ell} F(\bar{z}, z , \lambda) A_{\ell}^{-1}.
\end{equation*}
We obtain the following proposition.
\begin{Proposition}
    Let $\xi$ be the radially symmetric potential defined in \eqref{eq: Smyth potential} and let $F_{\lambda} \in \Lambda\SU_{\sigma}$ for $\lambda \in \mathbb{S}^{1}$ be the extended frame obtained by $\xi$. The minimal Lagrangian surface $f^{\lambda}: M \to Q_{2}^{*}$ constructed by $( F_{\lambda}, F_{i\lambda} )$ admits discrete rotational symmetries:
    \begin{equation*}
    \hat{\mathfrak f}^{\lambda}( R_{\ell}(z), \overline{R_{\ell}(z)}, \lambda )  = \mathcal{A}_{\ell} \hat{\mathfrak f}^{\lambda}(\bar{z} , z , \lambda), \quad
    \mbox{with} \quad \mathcal{A}_{\ell} :=\begin{pmatrix}
 1 & 0 & 0 & 0 \\
 0 & 1 & 0 & 0 \\
 0 & 0 & \cos \left( \frac{2\pi \ell}{k + 2} \right) & \sin \left( \frac{2\pi \ell}{k + 2} \right) \\
 0 & 0 & -\sin \left( \frac{2\pi \ell}{k + 2} \right) & \cos \left( \frac{2\pi \ell}{k + 2} \right)
\end{pmatrix},
\end{equation*}
 where $\hat{\mathfrak f}^{\lambda}$ denotes the horizontal lift of $f^{\lambda}$.
 Moreover, the induced metric of $f^{\lambda}$ depends only on the radial coordinate $|z|$. 
Such a surface $f^{\lambda}$ is therefore called {\rm radially symmetric}.
\end{Proposition}
\begin{proof}
    Let $\psi : \Lambda \SU_{\sigma} \times \Lambda \SU_{\sigma} \to \Lambda \SO_{\sigma}$ be a loop group homomorphism, and 
    set $\hat{\mathcal F}(z, \bar z, \lambda) = \psi (F(z, \bar z, \lambda), F(z, \bar z,  i \lambda))$.
    By direct computation, we have
\begin{equation*}
      \hat{\mathcal F}( R_{\ell}(z), \overline{R_{\ell}(z)}, \lambda ) = \psi \left(  A_{\ell} , A_{\ell}  \right) \cdot \hat{\mathcal F}(\bar{z} , z , \lambda) \cdot  \left( \psi  \left( A_{\ell} , A_{\ell} \right) \right)^{-1},
\end{equation*}
where 
\begin{equation*}
    \psi\left( A_{\ell}, A_{\ell} \right) = \begin{pmatrix}
 \mathrm{Id}_{2\times2} & 0 \\
 0 & A
\end{pmatrix}, \quad A = \begin{pmatrix}
 \cos \left( \frac{2\pi \ell}{k + 2} \right) & \sin \left( \frac{2\pi \ell}{k + 2} \right) \\
 -\sin \left( \frac{2\pi \ell}{k + 2} \right) & \cos \left( \frac{2\pi \ell}{k + 2} \right)
\end{pmatrix} \in \mathrm{SO}(2).
\end{equation*}
 Then a straightforward computation shows the first claim. 

As shown in \cite[Proposition 6.2]{BRS2010}, the solution $\hat{u}$ of the elliptic sinh-Gordon
equation \eqref{eq: sinh-Gordon equation} for a surface generated by $\xi$ in \eqref{eq: Smyth potential}, with $\Phi(z_0) = \id$, depends only on $|z|$. Hence, by \eqref{eq: metric}, the solution $u$ of \eqref{eq: second order PDE} associated with the minimal Lagrangian surface $f^{\lambda}$ depends only on $|z|$. This completes the proof.
 \end{proof}

\appendix
\section{Comparison to the result of Gao-Van der Veken-Wijffels-Xu}\label{app: CU}

In this section, we shall discuss a relation to the result of Gao-Van der Veken-Wijffels-Xu \cite{GVWX} and establish a correspondence with their quantities. 

Let $\Phi = \left( \phi, \psi \right) : \Sigma \to \mathbb{H}^{2} \times \mathbb{H}^{2}$ be a Lagrangian immersion of an oriented surface with
area 2-form $\omega_{\Sigma}$. Define a function $\Gamma : \Sigma \to \mathbb{R}$ such that
\begin{equation*}
    \phi^{\ast}\omega_{0} = \psi^{\ast}\omega_{0} = \Gamma \omega_{\Sigma},
\end{equation*}
where $\omega_{0}$ is the K\"ahler 2-form defined in $\mathbb{H}^{2}$ \cite{GVWX}.
To compare $f$ and $\Phi$, we consider a local isothermal parameter $z = x + iy$ on $\Sigma$. Since $Q_2^{*}$ is isometric to $\mathbb H^2 \times \mathbb H^2$, where 
 two hyperbolic planes $\mathbb H^2$ have the constant curvature $-4$ \cite{WV2021}, we use the metric 
  \[
  ds_{\Sigma}^2 = 4 ds^{2}_{M} = 8e^{u}\mathrm{d}z\mathrm{d}\bar{z}.
 \]
The formulas for the minimal Lagrangian immersion $\Phi$ in $\mathbb{H}^{2} \times \mathbb{H}^{2}$ transform into 
\begin{align}
    &\Phi_{z\bar{z}} = 2e^{u}\Phi, \label{eq: GVWX 1} \\
    &16e^{2u}(1 - 4\Gamma^{2}) = \left|\langle \Phi_{z} , \hat{\Phi}_{z} \rangle\right|^{2} = 4\left| \left\langle \phi_{z} , \phi_{z} \right\rangle \right|^{2} = 4\left| \left\langle \psi_{z}, \psi_{z}  \right\rangle \right|^{2}, \label{eq: GVWX 4}
\end{align}
where the derivatives with respect to $z$ and $\bar{z}$ are given by $\partial_z = \tfrac12 (\partial_x - i \partial_y)$, $\partial_{\bar z} = \tfrac12 (\partial_x + i \partial_y)$, 
and $\hat{\Phi} = (\phi , -\psi)$. 
The equation \eqref{eq: GVWX 1} implies that $\phi_{z\bar{z}} = 2e^{u}\phi$ and $\psi_{z\bar{z}} = 2e^{u}\psi$. This means that $\phi, \psi : \Sigma \to \mathbb{H}^{2}$ are harmonic maps. Thus, the associated Hopf differential $\Theta(z)  := (1/2) \langle \Phi_{z}, \hat{\Phi}_{z} \rangle \otimes (\mathrm{d} z)^{2}$ is holomorphic. From \eqref{eq: GVWX 4}, we have the Codazzi equation
\begin{equation}
    |\Theta|^{2} = 4e^{2u} \left(1 - 4 \Gamma^{2}\right).
\end{equation}
 By Lemma 5.3 in \cite{GVWX} and the Gauss curvature $K = -\frac14 e^{-u}u_{z\bar{z}}$, we obtain the Gauss equation
 \begin{equation}\label{eq:generalGC1}
4|\Gamma_{z}|^{2}= \left(1- 4 \Gamma^2\right) \left(u_{z\bar{z}}  - 8e^{u} \Gamma^{2}\right). 
\end{equation}        
By setting $\hat \Gamma = \frac{1}{2}e^{-u} |\beta|$ and a straightforward computation, it follows that \eqref{eq:generalGC1}  is equivalent to \eqref{eq: second order PDE}. Then it is natural to have the 
 correspondence of the associated Hopf differential $\Theta$ and the function $\Gamma$ 
with the quantities $\alpha, \beta$ and $u$ by
\begin{equation}\label{eq: CU alpha}
    \Theta = 2\alpha, \quad     \Gamma  = \hat \Gamma = \frac{1}{2}e^{-u} |\beta|.
\end{equation}
 Indeed, a straightforward computation shows that \eqref{eq: CU alpha} holds for surfaces 
 $\Phi$ in $\mathbb H^2  \times \mathbb H^2$ and  $f$ in $Q_2^{*}$.
 
On behalf of all authors, the corresponding author states that there is no conflict of
interest. No new data were created or analyzed in this study.

\bibliographystyle{plain}
\bibliography{mybib}

\end{document}